\documentclass[10pt]{amsart}
\usepackage{amssymb}
\usepackage{dsfont}
\usepackage{amscd}
\usepackage[mathscr]{euscript} % for the power-set P
\usepackage[all]{xy}
\usepackage{url}
\usepackage{stmaryrd}
\usepackage{comment}
\usepackage[retainorgcmds]{IEEEtrantools}
\usepackage{hyperref}
\title[Model theory of Galois actions of torsion Abelian groups]{Model theory of Galois actions of torsion Abelian groups}
\author[\"{O}. BEYARSLAN]{\"{O}zlem Beyarslan$^{\clubsuit}$}
\thanks{$^{\clubsuit}$ Supported by the T\"{u}bitak 1001 grant no. 119F397.}
\address{$^{\clubsuit}$Bo\v{g}azi\c{c}i \"{U}niversitesi}
\email{ozlem.beyarslan@boun.edu.tr}
\author[P. KOWALSKI]{Piotr Kowalski$^{\spadesuit}$}
\thanks{$^{\spadesuit}$ Supported by the Narodowe Centrum Nauki grant no. 2018/31/B/ST1/00357 and by the T\"{u}bitak 1001 grant no. 119F397.}
\address{$^{\spadesuit}$Instytut Matematyczny\\
Uniwersytet Wroc{\l}awski\\
Wroc{\l}aw\\
Poland}
\email{pkowa@math.uni.wroc.pl} \urladdr{http://www.math.uni.wroc.pl/\textasciitilde pkowa/ }
\thanks{2010 \textit{Mathematics Subject Classification} Primary 03C60; Secondary 12H10, 11S20, 20E18.}
\thanks{\textit{Key words and phrases}. Difference field, Model companion, Pr\"{u}fer group, Frattini cover.}

\DeclareMathOperator{\cl}{cl}
  
\DeclareMathOperator{\gl}{GL} \DeclareMathOperator{\aut}{Aut} \DeclareMathOperator{\id}{id}
 \DeclareMathOperator{\fr}{Fr} 
  \DeclareMathOperator{\fix}{Fix}
  \DeclareMathOperator{\gal}{Gal}

 \DeclareMathOperator{\theo}{Th}\DeclareMathOperator{\alg}{alg}

\DeclareMathOperator{\coli}{\underrightarrow{\lim}}

\DeclareMathOperator{\li}{\underleftarrow{\lim}}

\DeclareMathOperator{\spec}{Spec}\DeclareMathOperator{\rat}{rat}
\DeclareMathOperator{\sep}{sep}

\DeclareMathOperator{\dcf}{DCF}

\newtheorem{theorem}{Theorem}[section]
\newtheorem{prop}[theorem]{Proposition}
\newtheorem{lemma}[theorem]{Lemma}
\newtheorem{cor}[theorem]{Corollary}
\newtheorem{fact}[theorem]{Fact}
\newtheorem{conj}[theorem]{Conjecture}
\theoremstyle{definition}
\newtheorem{definition}[theorem]{Definition}
\newtheorem{example}[theorem]{Example}
\newtheorem{remark}[theorem]{Remark}

\begin{document}
\newcommand{\lili}{\underleftarrow{\lim }}
\newcommand{\coco}{\underrightarrow{\lim }}
\newcommand{\twoc}[3]{ {#1} \choose {{#2}|{#3}}}
\newcommand{\thrc}[4]{ {#1} \choose {{#2}|{#3}|{#4}}}
\newcommand{\Zz}{{\mathds{Z}}}
\newcommand{\Ff}{{\mathds{F}}}
\newcommand{\Cc}{{\mathds{C}}}
\newcommand{\Rr}{{\mathds{R}}}
\newcommand{\Nn}{{\mathds{N}}}
\newcommand{\Qq}{{\mathds{Q}}}
\newcommand{\Kk}{{\mathds{K}}}
\newcommand{\Pp}{{\mathds{P}}}
\newcommand{\ddd}{\mathrm{d}}
\newcommand{\Aa}{\mathds{A}}
\newcommand{\dlog}{\mathrm{ld}}
\newcommand{\ga}{\mathbb{G}_{\rm{a}}}
\newcommand{\gm}{\mathbb{G}_{\rm{m}}}
\newcommand{\gaf}{\widehat{\mathbb{G}}_{\rm{a}}}
\newcommand{\gmf}{\widehat{\mathbb{G}}_{\rm{m}}}
\newcommand{\ka}{{\bf k}}
\newcommand{\ot}{\otimes}
\newcommand{\si}{\mbox{$\sigma$}}
\newcommand{\ks}{\mbox{$({\bf k},\sigma)$}}
\newcommand{\kg}{\mbox{${\bf k}[G]$}}
\newcommand{\ksg}{\mbox{$({\bf k}[G],\sigma)$}}
\newcommand{\ksgs}{\mbox{${\bf k}[G,\sigma_G]$}}
\newcommand{\cks}{\mbox{$\mathrm{Mod}_{({A},\sigma_A)}$}}
\newcommand{\ckg}{\mbox{$\mathrm{Mod}_{{\bf k}[G]}$}}
\newcommand{\cksg}{\mbox{$\mathrm{Mod}_{({A}[G],\sigma_A)}$}}
\newcommand{\cksgs}{\mbox{$\mathrm{Mod}_{({A}[G],\sigma_G)}$}}
\newcommand{\crats}{\mbox{$\mathrm{Mod}^{\rat}_{(\mathbf{G},\sigma_{\mathbf{G}})}$}}
\newcommand{\crat}{\mbox{$\mathrm{Mod}^{\rat}_{\mathbf{G}}$}}
\newcommand{\cratinv}{\mbox{$\mathrm{Mod}^{\rat}_{\mathbb{G}}$}}
\newcommand{\ra}{\longrightarrow}
\newcommand{\bdcf}{B-\dcf}
\makeatletter
\providecommand*{\cupdot}{%
  \mathbin{%
    \mathpalette\@cupdot{}%
  }%
}
\newcommand*{\@cupdot}[2]{%
  \ooalign{%
    $\m@th#1\cup$\cr
    \sbox0{$#1\cup$}%
    \dimen@=\ht0 %
    \sbox0{$\m@th#1\cdot$}%
    \advance\dimen@ by -\ht0 %
    \dimen@=.5\dimen@
    \hidewidth\raise\dimen@\box0\hidewidth
  }%
}

\providecommand*{\bigcupdot}{%
  \mathop{%
    \vphantom{\bigcup}%
    \mathpalette\@bigcupdot{}%
  }%
}
\newcommand*{\@bigcupdot}[2]{%
  \ooalign{%
    $\m@th#1\bigcup$\cr
    \sbox0{$#1\bigcup$}%
    \dimen@=\ht0 %
    \advance\dimen@ by -\dp0 %
    \sbox0{\scalebox{2}{$\m@th#1\cdot$}}%
    \advance\dimen@ by -\ht0 %
    \dimen@=.5\dimen@
    \hidewidth\raise\dimen@\box0\hidewidth
  }%
}
\makeatother

\def\Ind#1#2{#1\setbox0=\hbox{$#1x$}\kern\wd0\hbox to 0pt{\hss$#1\mid$\hss}
\lower.9\ht0\hbox to 0pt{\hss$#1\smile$\hss}\kern\wd0}

\def\ind{\mathop{\mathpalette\Ind{}}}

\def\notind#1#2{#1\setbox0=\hbox{$#1x$}\kern\wd0
\hbox to 0pt{\mathchardef\nn=12854\hss$#1\nn$\kern1.4\wd0\hss}
\hbox to 0pt{\hss$#1\mid$\hss}\lower.9\ht0 \hbox to 0pt{\hss$#1\smile$\hss}\kern\wd0}

\def\nind{\mathop{\mathpalette\notind{}}}

%\mbox{\rule{8pt}{8pt}}\vspace{0.3cm}\newline}
%\maketitle

\maketitle
\begin{abstract}
We show that the theory of Galois actions of a torsion Abelian group $A$ is companionable if and only if for each prime $p$, the $p$-primary part of $A$ is either finite or it coincides with the Pr\"{u}fer $p$-group. We also provide a model-theoretic description of the model companions we obtain.
\end{abstract}

\tableofcontents

\section{Introduction}
A \emph{difference field} is a field with an endomorphism. If this endomorphism is invertible, then a difference field is the same as an action of the group $(\Zz,+)$ by field automorphisms. Model theory of difference fields has been extensively studied for more than 20 years (see e.g. \cite{Mac1, acfa1, acfa2}). It is also natural to study model theory of fields with actions of an \emph{arbitrary} (fixed) group, instead of the infinite cyclic group. This topic had not been considered much until recently, we give a short account of earlier works below.
\begin{itemize}
\item Besides the theory ACFA corresponding to the action of $\Zz$, model theory of fields with free group actions has been also considered, which resulted in the theory ACFA$_n$, see e.g. \cite{Hr9}, \cite{KiPi}, \cite[Theorem 16]{Sjo}, and \cite[Proposition 4.12]{MS2}.

\item Actions of the group $\Zz\times \Zz$ were considered by Hrushovski, who proved that a model companion does not exist in this case (see \cite{Kikyo1}).

\item Actions of finite groups were considered first by Sj\"{o}gren in \cite{Sjo}.

\item Model theory of actions of $(\Qq,+)$ on fields were studied in \cite{med1}.
\end{itemize}
For a fixed group $G$, the first natural question to be considered is the following: does a model companion of the theory of fields with $G$-actions exist? In the examples given above, the corresponding model companions exist, except for the case of the group $\Zz\times \Zz$. If such a model companion exists, then we call this model companion $G$-TCF and we say that ``$G$-TCF exists''.

More recently, Daniel Hoffmann and the second author considered in \cite{HK3} the case of finite groups (being unaware then of Sj\"{o}gren's work from \cite{Sjo}). In \cite{BK}, the authors of this paper extended some of the results from \cite{acfa1} and \cite{HK3} into a very natural common context of \emph{virtually free} groups. This work is a continuation of the general line of research from \cite{BK}, however, it goes in a different direction, that is we consider infinite torsion Abelian groups. Let $A$ be a torsion Abelian group. This paper is almost exclusively devoted to the proof of the following result.
\begin{theorem}\label{superthm}
The theory $A-\mathrm{TCF}$ exists if and only if for each prime $p$, the $p$-primary part of $A$ is either finite or it is isomorphic with the Pr\"{u}fer $p$-group. Moreover, if the theory $A-\mathrm{TCF}$ exists, then it is simple; and it is strictly simple (that is: simple, not stable, and not supersimple) when $A$ is infinite.
\end{theorem}
Regarding the question of the existence of the theory $G$-TCF for $G$ coming from a given class of groups, the theorem above is a rare instance of a situation when a \emph{full} answer is given. For example, we are still far from obtaining a corresponding answer for the class of all finitely generated groups: we showed in \cite{BK} that for virtually free groups the corresponding model companions exist, and we only conjectured in \cite{BK} the opposite implication (this conjecture is confirmed in \cite{BK} in the case of finitely generated commutative groups). It should be also noted that Theorem \ref{superthm} disproves our own  \cite[Conjecture 5.12]{BK}, that is $G$-TCF need not exist for a locally virtually free (even locally finite) group $G$.

Let us fix our (standard) notation here. We denote the set of all prime numbers by $\Pp$. For $n>0$, the cyclic group of order $n$ is denoted by $C_n$  and
$$C_{p^{\infty}}=\coli_{n}C_{p^n}$$
is the Pr\"{u}fer $p$-group. For any group $G$ and any ordinal number $\alpha$, by $G^{(\alpha)}$ we denote the direct sum of $\alpha$ copies of $G$. If a group $G$ acts on a set $X$, then by $X^G$ we denote the set of invariants of this action. For a field $K$, $\gal(K)$ denotes the (profinite) absolute Galois group of $K$, that is the group $\gal(K^{\sep}/K)$, where $K^{\sep}$ is the separable closure of $K$.

By a \emph{$G$-field}, we mean a field with an action of the group $G$ by field automorphisms. Similarly, we consider $G$-rings, $G$-field extensions, etc. By $L_G$, we denote the natural language of $G$-fields, where the elements of $G$ serve as unary function symbols.

\begin{remark}\label{equivcond}
We give here two conditions on an Abelian group $A$, which are equivalent to the condition appearing in the statement of Theorem \ref{superthm}.
\begin{enumerate}
\item The group $A$ does not contain (up to an isomorphism) any of the following two ``forbidden subgroups'':
\begin{itemize}
\item $C_p^{(\omega)}$,

\item $C_p\oplus C_{p^{\infty}}$.
\end{itemize}

\item There is no $p\in \Pp$ such that there exists an infinite strictly increasing sequence
$$C_p^2\cong P_1<P_2<P_3<\ldots,$$
where each $P_i$ is a finite $p$-subgroup of $A$.
\end{enumerate}
%\textbf{Similarities with $\Zz\times \Zz$ and finitely generated groups, but, of course, there we don't have any implication (and we don't believe in one of %them).}
\end{remark}
We would like to say a few words about the shape of the axioms of the theory $A$-TCF from the statement of Theorem \ref{superthm}. The axioms of ACFA from \cite{acfa1} are \emph{geometric}, that is they describe the intersections of algebraic varieties with the graph of the generic automorphism. In the case of a finite group, geometric axioms were used in \cite{HK3} as well, which was the main difference with the approach taken in \cite{Sjo}. Using the Bass-Serre theory, the geometric axioms from \cite{HK3} were ``glued'' in \cite{BK} to obtain geometric axioms for actions of arbitrary finitely generated virtually free groups. Let us now go back to the axioms from \cite{Sjo}. They are not geometric in the above sense, since they describe the properties of certain absolute Galois groups. This is why we call them \emph{Galois axioms} and we formalize this notion below. Before that, let us recall that a field $K$ is \emph{pseudo algebraically closed} (abbreviated \emph{PAC}), if each absolutely irreducible variety defined over $K$ has a $K$-rational point.
\begin{definition}\label{galaxioms}
We say that the theory of a $G$-field $K$ is \emph{axiomatised by Galois axioms}, if $G$ is the union of its finitely generated subgroups $(G_i)_{i\in I}$ (for convenience we assume that $0$ is a distinguished element of $I$ and $G_0=\{1\}$) such that the theory of the $G$-field $K$ is implied by the following statements:
\begin{enumerate}
\item the action of $G$ on $K$ is faithful (we say that the $G$-field $K$ is \emph{strict});

\item $K$ is a perfect field;

\item for each $i\in I$, $K^{G_i}$ is PAC;

\item for each $i\in I$, we have:
$$\gal\left(K^{G_i}\right)\cong \mathcal{G}_{i},$$
where $(\mathcal{G}_i)_{i\in I}$ is a fixed collection of small profinite groups.
\end{enumerate}
\end{definition}
Let us remark here that the definition of ``strict'' above extends \cite[Definition 2.2]{HK3} in the case of a finite group $G$.

Clearly, Items $(1)$ and $(2)$ are first-order. By \cite[Chapter 11.3]{FrJa}, Item $(3)$ is a first-order condition as well. Since the set of extensions of a field $F$ (inside a fixed algebraic closure of $F$) of
a fixed degree $n$ is ``uniformly definable'' in $F$ (see e.g. \cite[Remark 2.6(i)]{PiPolk}) and there are finitely many of them in the situation of Definition \ref{galaxioms} (by the smallness assumption), we see that Item $(4)$ is also first-order.
%\\
%\textbf{Any better reference [?2?, Fried-Jarden!!]?}
%\\
We would like to point out that in the case of a group $G$ which is not finitely generated, the field of constants $K^G$ is \emph{not} definable in the $G$-field $K$ (it is merely type-definable). Hence, there is not much chance for any statement about $\gal(K^G)$ to be first-order.

As we have said above, the theory $G-\mathrm{TCF}$ is axiomatised by Galois axioms for a finite $G$ (by \cite{Sjo}), which we will also point out in Proposition \ref{finitecr}. In this paper, we prove a version of this result for torsion Abelian groups satisfying the equivalent conditions from Remark \ref{equivcond}.

This paper is organized as follows. In Section \ref{secgen}, we collect some general results (originating often from \cite{Sjo}) about existentially closed $G$-fields and we also discuss Hrushovski's notion of pseudo-existentially closed $G$-fields. In Section \ref{secabsolute}, we show a crucial technical result about pronilpotency of certain absolute Galois groups. Section \ref{secneg} is concerned with the negative (or non-existence) results. More precisely, we show there the left-to-right implication from Theorem \ref{superthm}. Section \ref{secpos} is about the positive results, that is we show there the right-to-left implication from Theorem \ref{superthm}. In Section \ref{secmisc}, we collect several miscellaneous results and observations regarding the model theory of $G$-fields.

\section{General results about $G$-fields}\label{secgen}
In this section, we discuss Hrushovski's notion of pseudo-existentially closed $G$-fields and we also collect the results about existentially closed $G$-fields and PAC fields, which will be important in the sequel. We finish this section with some well-known results about chains of theories and we give examples of such chains in the case of group actions on fields.

\subsection{Pseudo-existentially closed $G$-fields and PAC fields}\label{pecsec}
The following notion we learnt from Udi Hrushovski (private communication). It originated from our attempts to show that if $G$ has a subgroup isomorphic to $\Zz\times \Zz$, then $G$-TCF does not exist (those attempts will be discussed in Section \ref{secudi}).
\begin{definition}
A $G$-field $F$ is \emph{pseudo-existentially closed} (abbreviated \emph{p.e.c.}), if for any $G$-field extension $F\subseteq K$ such that the pure field extension $F\subseteq K$ is regular, the $G$-field $F$ is existentially closed in the $G$-field $K$.
\end{definition}
We would like to point out that if $G=\{1\}$, then p.e.c. $G$-fields are exactly PAC fields (since being PAC is the same as being existentially closed in regular extensions, see \cite[Proposition 11.3.5]{FrJa}), which justifies the choice of the term ``pseudo-existentially closed'' above. We will also use in the sequel the abbreviation ``e.c.'' for ``existentially closed''. We recall that a $G$-field is \emph{$G$-closed}, if it has no proper algebraic (as pure fields) $G$-field extensions.

The crucial good property of p.e.c. $G$-fields is the result below, which is clearly false for e.c. $G$-fields (consider $H=\{1\}$). This result and its proof was pointed out to us by Udi Hrushovski.
\begin{prop}\label{udipec}
Suppose that $M$ is a p.e.c. $G$-field and $H\leqslant G$. Then $M$ is a p.e.c. $H$-field as well.
\end{prop}
\begin{proof}
Assume that $F$ is a $G$-field and $F\subseteq K$ is an $H$-field extension, which is regular (as an extension of pure fields). To conclude the proof, it is enough to construct a field extension of $K\subseteq L$ and an action of $G$ on $L$ such that $F\subseteq L$ is a $G$-field extension and $K\subseteq L$ is an $H$-field extension.

Let $W$ of be a set of representatives for $G/H$ such that $1\in W$. For each $r\in W$ we fix a set $rK$ such that:
\begin{enumerate}
\item[(i)] $1K=K$;

\item[(ii)] for all $r,r'\in W$, if $r\neq r'$, then $rK\cap r'K=F$;

\item[(iii)] there is a bijection (denoted by $r$ as well) $r:K\to rK$ such that for all $x\in F$, we have $r(x)=r\cdot x$ (the action of $G$ on $F$).
\end{enumerate}
For any $g\in G$ and $r\in W$, there are unique $r'\in W$ and $h\in H$ such that $gr=r'h$ and we define a bijection $g:rK\to r'K$ by the following commutative diagram:
\begin{equation*}
 \xymatrix{ rK \ar[rr]^{g} \ar[d]_{r^{-1}}   & &  r'K \\
K  \ar[rr]^{\cdot h}  & &   K \ar[u]_{r'},}
\end{equation*}
where $\cdot h$ comes from the given action of $H$ on $K$. Let us also define:
$$Z:=\bigcup_{r\in W}rK.$$
It is easy to see that the above diagram defines an action of $G$ on the \emph{set} $Z$, which extends the action of $H$ on $K$ and the action of $G$ on $F$. For each $r\in W$, we define a field structure on $rK$ in such a way that $r:K\to rK$ is a field isomorphism. Then for each $g,r,r'$ as above, the map $g:rK\to r'K$ is a field isomorphism as well.

We define now:
$$R:=\bigotimes_{r\in W}rK,$$
where the tensor product is taken over the field $F$. By the universal property of the tensor product (in the category of $F$-algebras), the action of $G$ on $Z$ uniquely extends to an action of $G$ on $R$ by ring automorphisms. Since the field extension $F\subseteq K$ is regular, the ring $R$ is a domain by \cite [Proposition 2 in \S 17, A.V.141]{bourbalgebra}. Hence, we can take as $L$ the field of fractions of $R$ with the induced action of $G$.
\end{proof}
\begin{remark}\label{pecobs}
We collect here several observations about e.c., p.e.c, and $G$-closed fields. Let $H\leqslant G$.
\begin{enumerate}
\item[(0)] Any $G$-closed
$G$-field is perfect \cite[proof of Theorem 1]{Sjo} and if $K$ is a perfect $G$-field then $K^G$
is also perfect \cite[Lemma 3.2]{HK3}. 

\item We will give a quick argument showing that a $G$-field $K$ is e.c. if and only if $K$ is both p.e.c. and $G$-closed. 
\\
The left-to-right implication is clear (two weakenings). For the right-to-left implication, we assume that $K$ is a $G$-closed $G$-field and $K\subseteq M$ is a $G$-field extension. It is enough to show that the field extension $K\subseteq M$ is regular. By Item $(0)$ above, $K$ is perfect, hence this extension is separable. Since $K$ is $G$-closed, $K$ is relatively algebraically closed in $M$ (note that this relative algebraic closure is a $G$-subfield of $M$), therefore the field extension $K\subseteq M$ is regular indeed.

\item It is also clear that if $K$ is a $G$-field, which is $H$-closed, then $K$ is $G$-closed as well.

\item Proposition \ref{udipec} says that the opposite happens with the notion of p.e.c:
if $K$ is a p.e.c. $G$-field, then $K$ is a p.e.c. $H$-field.

\item It is easy to find examples of e.c. $G$-fields, which are not e.c. $H$-fields (taking $H=\{1\}$).

\item We still do not know whether the existence of $G$-TCF implies the existence of $H$-TCF (see Conjecture \ref{zzconj} for a special case).

\item The notion of a p.e.c. $G$-field and Proposition \ref{udipec} should generalize to the context of an arbitrary theory  (instead of the theory of fields) having a stable completion, which is the context considered in \cite{Hoff3}.
\end{enumerate}
\end{remark}
The following result will be crucial in the sequel.
\begin{prop}\label{perfpac}
If $K$ is a p.e.c. $G$-field and $G$ is finitely generated, then the field $K^G$ (the field of constants) is PAC.
\end{prop}
\begin{proof}
Let us denote $K^G$ by $C$, and we take an absolutely irreducible variety $V$ over $C$. Then, the field extension $C\subseteq C(V)$ is regular (see \cite[Section 10.2]{FrJa}). Therefore, the ring $R:=C(V)\otimes_CK$ is a domain (again by \cite [Proposition 2 in \S 17, A.V.141]{bourbalgebra}). We define a $G$-ring structure on $R$ such that $K\subseteq R$ is a $G$-ring extension and $R^G=C(V)$ in the obvious way. Then $V(R^G)\neq \emptyset$, since there is an $R^G$-rational point corresponding to the identity map. Let $L$ be the fraction field of $R$. Then  the $G$-action on $R$ extends uniquely to a $G$-action on $L$ by field automorphisms.

Therefore, we have:
\begin{itemize}
\item the extension $K\subseteq L$ is a $G$-field extension;

\item the field extension $K\subseteq L$ is again regular, since $L\cong_KK(V_K)$ and $V_K$ is absolutely irreducible, where we consider the base change variety here: $V_K:=V\times_{\spec(C)}{\spec(K)}$;

\item the statement ``$V(L^G)\neq \emptyset$'' is first-order (since $G$ is finitely generated).
\end{itemize}
Since $K$ is a p.e.c. $G$-field, we obtain that $V(K^G)\neq \emptyset$, which finishes the proof.
\end{proof}
\begin{cor}\label{constpac}
Suppose that $K$ is a p.e.c $G$-field and $H\leqslant G$ is a finitely generated subgroup. Then the field $K^H$ is PAC. In particular, $K=K^{\{1\}}$ is a PAC field.
\end{cor}
\begin{proof}
It follows directly from Propositions \ref{udipec} and \ref{perfpac}.
\end{proof}
\begin{remark}
We would like to comment here how the results of this subsection are related to \cite{Sjo}.
\begin{enumerate}
\item It is stated in \cite[Theorem 2]{Sjo} that if $G$ is an \emph{arbitrary} group and $K$ is an existentially closed $G$-field, then $K^G$ is PAC. The proof of  \cite[Theorem 2]{Sjo} is basically the same as the proof of Proposition \ref{perfpac} above, and we believe that one still has to assume that $G$ is finitely generated for this proof to work (although, we do not have a counterexample for the statement with an arbitrary group $G$).

\item In \cite[Theorem 3]{Sjo}, it is stated that if $K$ is an existentially closed $G$-field, then $K$ is PAC. Corollary \ref{constpac} is more general and its proof is simpler, since it is using Proposition \ref{udipec}.
\end{enumerate}
 \end{remark}

\subsection{Existentially closed $G$-fields}\label{galax}
In this subsection, we go through several results which originally appeared in \cite{Sjo}, namely: Theorems 4, 5, and 6 there. We do it for the sake of completeness and
%several reasons. Firstly, the working assumption on the group $G$ in \cite{Sjo} is: finitely generated and finitely presented, which is not a satisfying %generality for us. Secondly,
we would like to discuss some issues concerning \cite[Theorem 6]{Sjo} as well.
%\begin{remark}
%It is {\bf really?} not explicitly written in \cite{Sjo}, but the finite generation/presentation is actually needed to axiomatize the class of $G$-closed %fields. Our arguments (later in the paper) probably yield that e.g. the class of $C_p^{\oplus \omega}$-closed fields is \emph{not} axiomatizable.
%\end{remark}

Suppose that $K$ is a $G$-field, $C=K^G$, and
$$\varphi:G\ra \aut(K/C)$$
corresponds to the action of $G$ on $K$.
%\begin{lemma}\label{reg}
%If $K$ is perfect, then the field extension $C^{\alg}\cap K\subseteq K$ is regular.
%\end{lemma}
%\begin{proof}
%Since the field $K$ is perfect, the field $C^{\alg}\cap K$ is also perfect, so the extension $C^{\alg}\cap K\subseteq K$ isegular.
%\end{proof}
For a group $G$, we denote by $\widehat{G}$ the profinite completion of $G$. For a profinite group $\mathcal{G}$, we denote by $\widetilde{\mathcal{G}}$ the universal Frattini cover of $\mathcal{G}$ (see \cite[Chapter 22]{FrJa}).  For a cardinal number $\kappa$ and $p\in \Pp$, we denote by $\widehat{F}_{\kappa}(p)$ the free pro-$p$ group of rank $\kappa$ (see \cite[Remark 17.4.7]{FrJa}).

The lemma below originates from \cite[Theorem 4]{Sjo}. It is also related to \cite[Proposition 5.1]{Hoff4}, where a version of this lemma is proved in a more general context (see Remark \ref{pecobs}(6)).
\begin{lemma}\label{contepi}
There is a unique continuous epimorphism:
$$\alpha:\widehat{G}\ra \gal\left(\left(C^{\alg}\cap K\right)/C\right)$$
such that the following diagram is commutative:
\begin{equation*}
 \xymatrix{ \widehat{G} \ar[rr]^{\alpha\ \ \ \ \ \ }  &  &  \gal(C^{\alg}\cap K/C) \\
G \ar[u]_{\iota} \ar[rr]^{\varphi\ \ \ \ \ \ } &  &   \aut(K/C) \ar[u]_{\mathrm{res}}.}
\end{equation*}
Moreover, if $(K,G)$ is p.e.c. and $G$ is finitely generated, then the map $\alpha$ is an isomorphism.
\end{lemma}
\begin{proof}
Let us consider a finite Galois extension $C\subseteq C'$ such that $C'\subseteq K$. We consider the restriction map:
$$r:G\ra \gal(C'/C).$$
It is enough to show that the map $r$ is onto. Let $H:=r(G)$ and $C_0:=(C')^{H}$. Then $G$ acts trivially on $C_0$, hence $C_0\subseteq C$ and $H=\gal(C'/C)$, which we needed to show.

For the moreover part, it is shown in the proof of \cite[Theorem 4]{Sjo} (in the e.c. case) that $G$ and $\gal(C^{\alg}\cap K/C)$ have the same finite quotients, which is enough (see \cite[Corollary 16.10.8]{FrJa}), since $\widehat{G}$ is small being topologically finitely generated, which follows from the assumption that $G$ is finitely generated. We sketch below the argument given in \cite{Sjo}. Let $\pi:G\to H$ be an epimorphisms and $\iota:H\to S_m$ be an embedding, where $S_m$ is the symmetric group on $m=|H|$ generators ($m$ is finite). Then the field of rational functions $K':=K(t_1,\ldots,t_m)$ has a natural structure of a $G$-extension of $K$ (given by $\iota\circ \pi$). It is easy to see that
$$C(t_1,\ldots,t_m)^H=\left(K'\right)^G,$$
hence the field of constants of $K'$ has a Galois extension with Galois group isomorphic to $H$. Since this last condition is first-order ($G$ is finitely generated), $K$ is p.e.c., and the extension $K\subset K'$ is regular, the result follows.
\end{proof}
We recall (see \cite[Definition 22.5.1]{FrJa})) that a continuous epimorphism of profinite groups $f:\mathcal{G}\to \mathcal{H}$ is a \emph{Frattini cover}, if for any closed subgroup $\mathcal{G}_0\leqslant \mathcal{G}$, we have that $\mathcal{G}_0=\mathcal{G}$ if and only if $f(\mathcal{G}_0)=\mathcal{H}$. We give here a connection between Frattini covers and extensions of group actions. It is just a reformulation of \cite[Lemma 3.7]{HK3} (see also \cite[Remark 3.8]{HK3}).
\begin{lemma}\label{fraextbest}
We assume that
\begin{itemize}
\item $C\subseteq K$ is a finite Galois extension and $G=\gal(K/C)$;

\item $\mathcal{G}_0\leqslant \gal(C)$ is a closed subgroup;

\item $C':=\left(C^{\sep}\right)^{\mathcal{G}_0}$;

\item $K':=C'K=\left(C^{\sep}\right)^{\mathcal{G}_0\cap \gal(K)}$.
\end{itemize}
Then the following are equivalent:
\begin{enumerate}
\item $K\subseteq K'$ is a $G$-field extension, where the $G$-field structure on $K'$ is given by $\gal(K'/C')$;

\item $\mathrm{res}(\mathcal{G}_0)=G=\gal(K/C)$, where the map
$$\mathrm{res}:\gal(C)\ra G$$
is the restriction epimorphism.
\end{enumerate}
Moreover, if any of the equivalent two conditions above holds, then we have:
$$[K':K]=[\gal(C):\mathcal{G}_0].$$
\end{lemma}
\begin{proof}
We just point out here that both the conditions are equivalent to the fact that the restriction map:
$$\mathrm{res}:\gal(K'/C')\ra \gal(K/C)$$
is an isomorphism, so $K'$ gets the $G$-field structure (extending the one on $K$) using this restriction isomorphism.
\end{proof}
We will use several times the following consequence of an implication from Lemma \ref{fraextbest}.
\begin{prop}\label{fraexcolim}
Suppose that we have a tower of fields $K_2\subseteq K_1\subseteq K$ such that $K/K_1$ and $K/K_2$ are finite Galois extensions and we set:
$$H:=\gal(K/K_1),\ \ \ \  G:=\gal(K/K_2).$$
Assume that for $i\in \{1,2\}$, $\mathcal{G}_i\leqslant\gal(K_i)$ are closed subgroups such that:
\begin{enumerate}
\item for $i\in \{1,2\}$, we have $\mathrm{res}(\mathcal{G}_i)=\gal(K/K_i)$,
%\item $\mathrm{res}(\mathcal{G}_2)=\gal(K/K_2)=G$,

\item $\mathcal{G}_1\cap \gal(K)=\mathcal{G}_2\cap \gal(K)$ and $\mathcal{G}_1\subseteq \mathcal{G}_2$.
\end{enumerate}
Let $K'$ (resp. $K''$) be the $H$-field (resp. $G$-field) extension of $K$ given by Lemma \ref{fraextbest}.
% and $K''$ be the $G$-field extension of $K$ given by Lemma \ref{fraextbest} as well.
Then $K'=K''$ and the $G$-structure on $K'$ expands the $H$-structure on $K'$.
\end{prop}
\begin{proof}
We define:
$$C':=\left(K^{\sep}\right)^{\mathcal{G}_1},\ \ \ \ C'':=\left(K^{\sep}\right)^{\mathcal{G}_2}.$$
Then we have:
$$K'=C'K,\ \ \ \ K''=C''K$$
and since $\mathcal{G}_1\cap \gal(K)=\mathcal{G}_2\cap \gal(K)$, we get that $K'=K''$.

We have the following commutative diagram:
\begin{equation*}
 \xymatrix{ \gal(K'/C'') \ar[rr]^{\mathrm{res}\ } \ar[rr]_{\cong\ } &  &  \gal(K/K_2) \\
\gal(K'/C') \ar[u]_{\leqslant}  \ar[rr]^{\mathrm{res}\ } \ar[rr]_{\cong\ }  & &  \gal(K/K_1) \ar[u]_{\leqslant}.}
\end{equation*}
By the description from Lemma \ref{fraextbest} of both the $G$-field structure and the $H$-field structure on $K'$, we see that the above diagram implies that the $G$-action expands the $H$-action.
\end{proof}
\begin{remark}
After assuming Item $(1)$ from Proposition \ref{fraexcolim}, Item $(2)$ there is equivalent to the following equality:
$$\mathcal{G}_2\cap \gal(K_1)=\mathcal{G}_1.$$
\end{remark}
%{\bf Below an older version, is it needed in such a form??}
%\begin{lemma}\label{fraext}
%We assume that
%\begin{itemize}
%\item $C\subseteq K$ is a finite Galois extension and $G=\gal(K/C)$;
%\item $K\subseteq L$ is a Galois extension (not necessarily finite) and $\mathcal{G}=\gal(L/C)$;
%\item $\pi:\mathcal{G}\to G$ is the restriction epimorphism;
%\item $\mathcal{G}_0\leqslant \mathcal{G}$ is closed such that $\pi(\mathcal{G}_0)=G$ and $G_0:=\mathcal{G}_0\cap \gal(L/K)$;
%\item $K':=L^{G_0}$ and $C':=L^{\mathcal{G}_0}$.
%\end{itemize}
%Then the restriction map to $\gal(K'/C')\to \gal(K/C)$ is an isomorphisms, hence the $G$-action on $K$ extends to $K'$. Moreover we have:
%$$[K':K]=[\mathcal{G}:\mathcal{G}_0].$$
%\end{lemma}
The next result generalizes an implication from Lemma \ref{fraextbest} (a version of it, in a more general context, appeared as \cite[Corollary 3.47]{Hoff3}).
\begin{prop}\label{fracov}
If $(K,G)$ is $G$-closed, then the restriction map
  $$\gal(C)\ra \gal(C^{\alg}\cap K/C)$$
  is a Frattini cover. Hence, if $C$ is PAC and $(K,G)$ is $G$-closed (for example, when $(K,G)$ is e.c. and $G$ is finitely generated), then this restriction map is the universal Frattini cover.
\end{prop}
\begin{proof}
Since $K$ is $G$-closed, we get by Remark \ref{pecobs}(0) that $C$ is a perfect field. Let us consider the restriction map:
$$\alpha:\gal(C)\ra \gal\left(C^{\alg}\cap K/C\right),$$
and we take a closed subgroup $\mathcal{G}\leqslant \gal(C)$ such that $\alpha(\mathcal{G})=\gal\left(C^{\alg}\cap K/C\right)$. It is enough to show that $\mathcal{G}=\gal(C)$. Since $\ker(\alpha)\mathcal{G}=\gal(C)$, we get that
$$\left(C^{\alg}\right)^{\ker(\alpha)}\cap \left(C^{\alg}\right)^{\mathcal{G}}=C.$$
Since the extension $C\subseteq \left(C^{\alg}\right)^{\ker(\alpha)}$ is Galois ($\ker(\alpha)$ is clearly a normal subgroup), we get that $(C^{\alg})^{\ker(\alpha)}$ is linearly disjoint from $(C^{\alg})^{\mathcal{G}}$ over $C$ using e.g. the remark below the proof of Corollary 2.5.2 in \cite{FrJa} (the remark is for finite extensions, but since both being Galois and being linearly disjoint are locally finite notions, it works in general). Since we have
$$C^{\alg}\cap K\subseteq \left(C^{\alg}\right)^{\ker(\alpha)},$$
 we get that $C^{\alg}\cap K$ is linearly disjoint from $(C^{\alg})^{\mathcal{G}}$ over $C$. Since $K$ is $G$-closed, it is perfect. Hence the field $C^{\alg}\cap K$ is perfect as well and the field extension $C^{\alg}\cap K\subseteq K$ is regular. Therefore (by the definition of regularity), $K$ is linearly disjoint from $C^{\alg}$ over $C^{\alg}\cap K$. By the Tower Property for linear disjointness (see \cite[Lemma 2.5.3]{FrJa}), we finally get that $K$ is linearly disjoint from $(C^{\alg})^{\mathcal{G}}$ over $C$. Therefore, we have
$$K(C^{\alg})^{\mathcal{G}}\cong \mathrm{Frac}\left(K\otimes_C(C^{\alg})^{\mathcal{G}}\right),$$
hence the action of $G$ on $K$ extends to an action of $G$ to $K(C^{\alg})^{\mathcal{G}}$. Since $(K,G)$ is $G$-closed, we get that $K(C^{\alg})^{\mathcal{G}}=K$, so $(C^{\alg})^{\mathcal{G}}\subseteq C^{\alg}\cap K$ and then, by linear disjointness, $(C^{\alg})^{\mathcal{G}}=C$. By Galois theory, we get that $\mathcal{G}=\gal(C)$, which finishes the proof.
\end{proof}
\begin{remark}
It is easy to see that the opposite implication to the one appearing in Proposition \ref{fracov} is not true. It is enough to take an algebraically closed field $C$, $K=C(X)$, and $G=\Zz$ acting on $K$ in the ``classical difference way'', that is $\sigma(X)=X+1$, where $\sigma$ is a generator of the group $\Zz$.
\end{remark}
Our first corollary is exactly \cite[Theorem 5]{Sjo} (it was also generalized to a more abstract context in \cite[Corollary 5.6]{Hoff4}).
\begin{cor}\label{constufc}
If $(K,G)$ is e.c. and $G$ is finitely generated, then we have:
$$\gal(C)\cong \widetilde{\widehat{G}}.$$
\end{cor}
\begin{proof}
It follows directly from Lemma \ref{contepi} and Proposition \ref{fracov}.
\end{proof}

The next corollary is much weaker than the statement in \cite[Theorem 6]{Sjo}, which will be discussed in Remark \ref{remove}(2).
\begin{cor}\label{monom}
Suppose that $(K,G)$ is e.c., then we have the following.
\begin{enumerate}
\item There is an epimorphism:
$$\gal(K)\ra \ker\left(\gal(C)\to \gal\left(C^{\alg}\cap K/C\right)\right).$$

\item If $G$ is finitely generated, then
$$\ker\left(\widetilde{\widehat{G}}\to \widehat{G}\right)\cong \ker\left(\gal(C)\to \gal\left(C^{\alg}\cap K/C\right)\right)$$
and there is a monomorphism:
$$\ker\left(\widetilde{\widehat{G}}\to \widehat{G}\right)\ra \gal(K).$$

\end{enumerate}
\end{cor}
\begin{proof}
Since the extension $C^{\alg}\cap K\subseteq K$ is regular (as in the proof of Proposition \ref{fracov}), the restriction map
$$r:\gal(K)\to \gal(C^{\alg}\cap K)$$
is onto.  By Galois theory, we have the following isomorphism:
$$\ker\left(\gal(C)\to \gal\left(C^{\alg}\cap K/C\right)\right)\cong  \gal\left(C^{\alg}\cap K\right)$$
showing Item $(1)$.

For Item $(2)$, by Corollary \ref{constufc} we have:
$$\ker\left(\widetilde{\widehat{G}}\to \widehat{G}\right)\cong \ker\left(\gal(C)\to \gal\left(C^{\alg}\cap K/C\right)\right).$$
Therefore, the profinite group $\gal(C^{\alg}\cap K)$ is projective (as a closed subgroup of the
profinite projective group $\widetilde{\widehat{G}}$, see \cite[Proposition 22.4.7]{FrJa}), hence the map $r$ above has a section, which gives the result.
\end{proof}
We point out below that for a finite group $G$, the theory $G-\mathrm{TCF}$ is axiomatized by Galois axioms, which was shown in \cite{Sjo} and \cite{HK3}. We include here a version of the statement from \cite{HK3}, which is convenient for us to work with. For the proof of this version, we need the following result, which may be folklore. We recall that for a profinite group $\mathcal{G}$, the \emph{rank} of $\mathcal{G}$, denoted $\mathrm{rk}(\mathcal{G})$, is the minimal cardinality of a set of topological generators of $\mathcal{G}$.
\begin{prop}\label{anyufc}
Assume that $\mathcal{H}$ is a profinite group of finite rank. Then, any continuous epimorphism $\pi:\widetilde{\mathcal{H}}\to \mathcal{H}$ is a (the universal) Frattini cover.
\end{prop}
\begin{proof}
The proof consists of two steps.

\vspace{10px}
\noindent
\textbf{Step 1}
\\
The result holds if $\mathcal{H}$ is a pro-$p$ group.
\begin{proof}[Proof of Step 1]
Let $r=\mathrm{rk}(\mathcal{H})$. By \cite[Lemma 22.7.4]{FrJa}, there is an epimorphism $\mathcal{H}\to C_p^r$ (which is a Frattini cover). By \cite[Lemma 22.5.4]{FrJa}, we can assume that $\mathcal{H}=C_p^r$.

Let us take $B\subset C_p^r$ such that $|B|=r$. Since $r$ is finite, $B$ generates the group $C_p^r$ if and only if $B$ is a basis of $C_p^r$ considered as an $\Ff_p$-vector space. Therefore, the group $\aut(C_p^r)=\gl_r(\Ff_p)$ acts transitively on the family of sets of generators of $C_p^r$ of size $r$.

By  \cite[Corollary 22.5.3]{FrJa}, $\mathrm{rk}(\widetilde{\mathcal{H}})=r$ and we fix $\widetilde{B}$, which is a set of generators of $\widetilde{\mathcal{H}}$ of size $r$. Let $B$ be the image of $\widetilde{B}$ by the universal Frattini cover map and $B':=\pi(\widetilde{B})$. Then, both $B$ and $B'$ have size $r$ and generate $C_p^r$. Hence, there is an automorphism of $C_p^r$ taking $B$ to $B'$. Therefore, $\pi$ is the composition of the universal Frattini cover map and this last automorphism, thus $\pi$ is a Frattini cover itself.
\end{proof}
Let us take a closed subgroup $\mathcal{G}\leqslant \widetilde{\mathcal{H}}$ such that
$$\pi|_{\mathcal{G}}:\mathcal{G}\ra \mathcal{H}$$
is a Frattini cover (it exists by  \cite[Lemma 22.5.6]{FrJa}). We aim to show that $\mathcal{G}=\widetilde{\mathcal{H}}$. For necessary background regarding profinite Sylow theory, we refer the reader to the beginning of Section \ref{secabsolute} (see also \cite[Proposition 22.9.2]{FrJa}).

\vspace{10px}
\noindent
\textbf{Step 2}
\\
For each $p\in \Pp$, any $p$-Sylow subgroup of $\mathcal{G}$ is a  $p$-Sylow subgroup of $\widetilde{\mathcal{H}}$ as well.
\begin{proof}[Proof of Step 2]
Let $\mathcal{Q}$ be a $p$-Sylow subgroup of $\mathcal{G}$. By \cite[Proposition 22.9.2(a)]{FrJa}, we have  $\mathcal{Q} = \mathcal{P} \cap \mathcal{G}$ for a $p$-Sylow subgroup $\mathcal{P}$ of $\widetilde{\mathcal{H}}$. Since
$\pi(\mathcal{G}) = \mathcal{H}$, we get that $\pi(\mathcal{Q})$ and $\pi(\mathcal{P})$ are p-Sylow subgroups of $\mathcal{H}$ (see \cite[Proposition 22.9.2(c)]{FrJa}), and thus
$\pi(\mathcal{Q})=\pi(\mathcal{P})$. By Step 1, the map:
$$\pi|_{\mathcal{P}}:\mathcal{P}\ra \pi(\mathcal{P})$$
is a Frattini cover and then we obtain that $\mathcal{Q}=\mathcal{P}$ (see \cite[Definition 22.5.1]{FrJa}).
\end{proof}
By Step 2, for any normal open subgroup $\mathcal{N}$ of $\widetilde{\mathcal{H}}$, any $p$-Sylow subgroup of $\mathcal{G}\mathcal{N}/\mathcal{N}$ is of the form $\mathcal{P}\mathcal{N}/\mathcal{N}$ for
a $p$-Sylow subgroup $\mathcal{P}$ of $\widetilde{\mathcal{H}}$, and thus it is a $p$-Sylow subgroup of $\widetilde{\mathcal{H}}/\mathcal{N}$. It follows that $\mathcal{G}\mathcal{N}/\mathcal{N} = \widetilde{\mathcal{H}}/\mathcal{N}$
for any $\mathcal{N}$ as above, and therefore we obtain that $\mathcal{G} = \widetilde{\mathcal{H}}$.
\end{proof}
\begin{remark}\label{unfortun}
\begin{enumerate}
\item The statement of Step 1 from the proof of Proposition \ref{anyufc} has already appeared in \cite{BK}  as Lemma 4.4. Unfortunately, we gave an erroneous proof of \cite[Lemma 4.4]{BK} (confusing two universal properties).

\item In the case of pro-$p$ groups, Proposition \ref{anyufc} can be generalized to the following statement, which we will need in the sequel:
$$\text{\textsl{Any continuous epimorphism of pro-$p$ groups of the same finite rank is a Frattini cover}.}$$
It can be proved in the same way as Step 1 in the proof of  Proposition \ref{anyufc} was shown. Namely, if $\pi:\mathcal{G}\to \mathcal{H}$ is a such an epimorphism and the rank is $r$, then, by \cite[Lemma 22.5.4]{FrJa}, there are Frattini cover maps:
    $$p_1:\mathcal{G}\ra C_p^r,\ \ \ \ p_2:\mathcal{H}\ra C_p^r.$$
    Hence, we can replace $\pi$ with $p_2\circ \pi$ and conclude as in the proof of Step 1 above.
\item It is easy to see that not every epimorphism of profinite (even finite) groups of the same finite rank is a Frattini cover, consider for example the following epimorphism:
$$C_6^2\ra C_6\times C_3.$$
\end{enumerate}
\end{remark}
We now point out that for a finite group $G$, the theory $G$-TCF is axiomatised by Galois axioms from Definition \ref{galaxioms}.
\begin{prop}\label{finitecr}
If $G$ is finite, then a $G$-field $K$ is existentially closed if and only if $K$ is strict, perfect, the field of constants $C:=K^G$ is PAC, and
$$\gal(C)\cong \widetilde{G}.$$
\end{prop}
\begin{proof}
For the left-to-right implication, we notice first that $K$ is strict by \cite[Lemma 3.4]{HK3} (see also 
the remark right after Definition \ref{galaxioms}). Using Remark \ref{pecobs}(0) and Remark \ref{pecobs}(1), we obtain that $K$ is perfect. By Proposition \ref{perfpac}, $C$ is PAC and by Corollary \ref{constufc}, we get that $\gal(C)\cong \widetilde{G}$.

For the right-left implication, we observe that $C$ is also perfect (since $K$ is perfect and the extensions $C\subseteq K$ is finite) and $C\subseteq K$ is a finite Galois extension such that $\gal(K/C)\cong G$ (since $G$ is finite and $K$ is
strict). Therefore, the pair $(C,K)$ satisfies the
assumptions of \cite[Theorem 3.29]{HK3}. By Proposition \ref{anyufc} and the assumption $\gal(C)\cong \widetilde{G}$, the restriction map
$\gal(C)\to \gal(K/C)$ is a Frattini cover, thus the pair $(C, K)$ is $G$-closed in the sense of \cite[Definition 3.18]{HK3} (or, equivalently, the $G$-field $K$ is $G$-closed by Lemma \ref{fraextbest}). We conclude the proof by
using the implication $(4)\Rightarrow (1)$ of \cite[Theorem 3.29]{HK3}.
\end{proof}
\begin{remark}\label{remove}
We would like to point out several general observations concerning Galois axioms and absolute Galois groups.
\begin{enumerate}
\item The original theory ACFA($=\Zz-\mathrm{TCF}$) is \emph{not} axiomatized by Galois axioms. To see that, we notice first that the Galois axioms in the case of a difference field $(K,\sigma)$ say that $K$ is algebraically closed and $C=\mathrm{Fix}(\sigma)$ is pseudofinite.

    By \cite[Section 13.3]{HrFro}, any model of ACFA of characteristic $0$ has infinite transcendence degree over $\Qq$.
%Section 5.1 in \url{http://www.math.ens.fr/~zchatzid/papiers/Helsinki.pdf}
By \cite[Theorem 18.5.6 and Theorem 18.6.1]{FrJa}, for almost all (in the sense of the Haar measure) $\sigma\in \gal(\Qq)$, the field $\fix(\sigma)$ is pseudofinite. Hence, such a difference field $(\Qq^{\alg},\sigma)$ satisfies the Galois axioms, but it is not existentially closed.
%Probably, it is always true that for a any field $C$, we can extend the $C$-algebra endomorphism induced by $X\mapsto X^2$ on $C(X)$ to $C(X)^{\alg}$ in %such a way that the resulting difference field $(C(X)^{\alg},\sigma)$ is \emph{not} a model of ACFA. Taking $C$ pseudofinite, we get a difference field %satisfying the ``Galois axioms'' for $G=\Zz$ but still not a model of $\Zz-\mathrm{TCF}=$ACFA.
%It is written in Section 13.3 (paragraph ``The examples of Cherlin-Jarden'') that: ``any model of ACFA of characteristic 0 must have infinite transcendence degree over $\Qq$.''.

\item It is stated in \cite[Theorem 6]{Sjo} that if $G$ is finitely generated and finitely presented and $K$ is an e.c. $G$-field, then we have
$$\gal(K)\cong \ker\left(\widetilde{\widehat{G}}\ra \widehat{G}\right).$$
The main part of the proof of \cite[Theorem 6]{Sjo} is an argument, which is supposed to show that the monomorphism appearing in Corollary \ref{monom} is actually an isomorphism. We do not know how to make this argument work, we comment more on it below.
\begin{enumerate}
\item The monomorphism from Corollary \ref{monom} is an isomorphism in the case of a finite group $G$, which was shown in \cite[Theorem 3.40(2)]{HK3}.

\item In \cite[Section 4]{BK}, we use \cite[Theorem 6]{Sjo} to show \cite[Theorem 4.7]{BK} saying that if $G$ is a finitely generated virtually free group, which is neither free nor finite, then the theory $G$-TCF is not simple, since the absolute Galois group of underlying fields of models of $G$-TCF are not small. However, if a profinite groups is small, then its image by a continuous epimorphism is also small. Therefore, in order to show \cite[Theorem 4.7]{BK}, it is enough to use just Corollary \ref{monom} instead of \cite[Theorem 6]{Sjo}.

\item Nick Ramsey communicated to us a proof of the result saying that for $G$ finitely generated and virtually free, the theory $G$-TCF is NSOP$_1$. However, this proof seems to be using the \emph{full} version of \cite[Theorem 6]{Sjo}.

\item Example \ref{cpinftyex}(2) gives a counterexample for the isomorphism:
$$\gal(K)\cong \ker\left(\widetilde{\widehat{G}}\ra \widehat{G}\right)$$
in the case of $G=C_{p^{\infty}}$, which is obviously not finitely generated.

\item Similar issues were discussed in a more general context in \cite{Hoff4} (see \cite[Conjecture 5.7]{Hoff4} and \cite[Remark 5.8]{Hoff4}).
\end{enumerate}
\end{enumerate}
\end{remark}

\subsection{Chains of theories}
In this subsection, we collect several well-known results about chains of theories. They can be found e.g. in \cite{med1} or \cite{KaPie}, but we include them here for the sake of completeness.

Let $L$ be a language and $T$ be an $L$-theory. It is easy to see that $T$ is \emph{closed under consequences} (that is: $T\models \phi$ if and only if $\phi\in T$) if and only if
$$T=\bigcap_{M\models T}\theo(M).$$
From now on, all the theories we consider are closed under consequences.

Suppose that $L\subseteq L'$ are languages, $T$ is an $L$-theory, and $T'$ is an $L'$-theory. We have the following obvious result.
\begin{fact}\label{inter1}
The following are equivalent.
\begin{enumerate}
  \item $T\subseteq T'$.

  \item ``$\mathrm{Mod}(T')\subseteq \mathrm{Mod}(T)$'', i.e. for each $M'\models T'$, we have $M'\models T$.
\end{enumerate}
\end{fact}
%\begin{proof}
%The implication $(1)\Rightarrow (2)$ is immediate.
%For $(2)\Rightarrow (1)$, we know by Fact \ref{obvious} that
%  $$T'=\bigcap_{M'\models T'}\theo(M').$$
%Hence, it is enough to show that for each $M'\models T'$, we have $T\subseteq \theo_L(M')$. But the last inclusion means exactly that $M'|_L\models T$.
%\end{proof}
For the general notion of a \emph{model companion} of the theory $T$, we refer the reader to \cite[Section 7.1]{Ho}.  If a model companion of $T$ exists, then it is unique and we denote it by $T^{\mathrm{mc}}$. All the particular theories considered in this paper are universal, hence inductive, and then a model companion exists if and only if the class of existentially closed
models of $T$ is elementary and then $T^{\mathrm{mc}}$ coincides with the axiomatisation of this class. The following result is crucial and it appeared  in \cite{med1} and \cite{KaPie}.
\begin{prop}\label{mcfact}
Suppose we have an increasing sequence of languages $(L_m)_{m>0}$ and an increasing sequence of $L_m$-theories $(T_m)_{m>0}$. If the model companions $(T^{\mathrm{mc}}_m)_m$ form an increasing sequence as well, then the model companion of $T_{\infty}:=\bigcup_m T_m$ exists and we have:
$$T_{\infty}^{\mathrm{mc}}=\bigcup_{m=1}^{\infty}T^{\mathrm{mc}}_m.$$
Moreover, if all the theories $T^{\mathrm{mc}}_m$ are simple, then the theory $T_{\infty}^{\mathrm{mc}}$ is simple as well.
\end{prop}
\begin{proof}
Let us denote:
$$L_{\infty}:=\bigcup_{m=1}^{\infty}L_m,\ \ \  T'_{\infty}:=\bigcup_{m=1}^{\infty}T^{\mathrm{mc}}_m.$$
It is easy to see and it is pointed out e.g. in \cite[Theorem 2]{med1} that $T'_{\infty}$ inherits all the ``local'' properties enjoyed by all the theories $T^{\mathrm{mc}}_m$. In particular, $T'_{\infty}$ is model complete and if all the theories $T^{\mathrm{mc}}_m$ are simple, then the theory $T'_{\infty}$ is simple as well. Therefore, it is enough to show that each model of $T_{\infty}$ embeds into a model of $T'_{\infty}$. We will actually show that this last embedding property is also ``local''.

Let us fix $M\models T_{\infty}$. We need to show that the theory $\mathrm{diag}^+(M)\cup T'_{\infty}$ is consistent, where $\mathrm{diag}^+(M)$ is the set of all atomic $L_{\infty}$-sentences with parameters from $M$ which are true in $M$. Since we have:
$$\mathrm{diag}^+(M):=\bigcup_{m=1}^{\infty}\mathrm{diag}^+\left(M|_{L_m}\right),$$
the result follows.
\end{proof}
\begin{example}\label{cpinftyex}
We give here an argument showing that for any $p\in \Pp$, the theory $C_{p^{\infty}}-\mathrm{TCF}$ exists, which may be considered as a ``baby case'' of the right-to-left implication in Theorem \ref{superthm}.
\begin{enumerate}
\item By Fact \ref{inter1} and Proposition \ref{mcfact}, it is enough to check that
if $(K,\sigma)\models C_{p^{m+1}}-\mathrm{TCF}$, then $(K,\sigma^p)\models C_{p^{m}}-\mathrm{TCF}$. Let $C=\fix(\sigma)$ and $C'=\fix(\sigma^p)$. Consider the following commutative diagram with exact rows:
\begin{equation*}
 \xymatrix{0 \ar[r]^{}  &  \gal(K) \ar[rr]^{<} & &   \gal(C)=\Zz_p  \ar[rr]^{\mathrm{res}\ \ \ \ }  &  &  \gal(K/C)=C_{p^{m+1}} \ar[r]^{} & 0 \\
0 \ar[r]^{}  &  \gal(K) \ar[u]_{=}\ar[rr]^{<} & &  \gal(C') \ar[u]_{<} \ar[rr]^{\mathrm{res}\ \ \ \ } &  &   \gal(K/C')=pC_{p^{m+1}}  \ar[u]_{<}  \ar[r]^{} & 0,}
\end{equation*}
where the description of the profinite group $\gal(C)$ comes from Proposition \ref{finitecr}. Hence, we have:
$$\gal\left(C'\right)=\mathrm{res}^{-1}\left(pC_{p^{m+1}}\right)=p\Zz_p\cong \Zz_p.$$
Since $C\subseteq C'$ is a finite field extension and $C$ is a (perfect) PAC field, then $C'$ is (perfect) and PAC as well \cite[Corollary 11.2.5]{FrJa}. Hence,
$(K,\sigma^p)$ satisfies the Galois axioms for the theory $C_{p^{m}}-\mathrm{TCF}$ by Proposition \ref{finitecr}.

\item By Item $(1)$, it is easy to see that if $K$ is an e.c. $C_{p^{\infty}}$-field, then we have:
$$\gal(K)\cong \Zz_p.$$

\item If $A$ is any divisible group, $K$ is an $A$-field, and $C=K^A$, then $C^{\alg}\cap K=C$, since there are no non-trivial homomorphisms from a divisible group into a profinite group. Therefore, for an $A$-closed field $K$, the extension $C\subseteq K$ is regular (since $C$ is
perfect and $C^{\alg}\cap  K = C$), hence $C$ is algebraically closed  (since, by linear
disjointness, $KC^{\alg}$ has a natural structure of an $A$-extension of $K$). In particular, if $K\models C_{p^{\infty}}-\mathrm{TCF}$, then $\gal(C)=1$, where $C=K^{C_{p^{\infty}}}$ is the field of absolute constants.
\end{enumerate}
\end{example}
\begin{example}\label{sectworem}
We discuss here what may happen if the model companions from Proposition \ref{mcfact} exist, but they do not form an increasing chain.
\begin{enumerate}
\item The theories $(C_{p^{m}}^2-\mathrm{TCF})_{m>0}$ do not form an increasing chain. To see that let us take
$$(K,\sigma,\tau)\models C_{p^{2}}^2-\mathrm{TCF},\ \ \ C=\fix(\sigma,\tau),\ \ \ C'=\fix(\sigma^p,\tau^p).$$
Then, we have:
$$\gal(C)/\gal(C')\cong \gal(C'/C)\cong C_p^2.$$
By Proposition \ref{finitecr} and \cite[Proposition 22.7.6]{FrJa} (this is a result of Tate saying that projective pro-$p$ groups are pro-$p$ free), we get $\gal(C)\cong \widehat{F}_2(p)$. Since $[\gal(C):\gal(C')]=p^2$, we get by \cite[Proposition 17.6.2]{FrJa} (the profinite Nielsen-Schreier formula)
 that
$$\gal\left(C'\right)\cong \widehat{F}_{p^2+1}(p)\ncong \gal(C).$$
In particular, by Proposition \ref{finitecr}, the $C_p^2$-field $(C',\sigma^p,\tau^p)$ is not existentially closed, so we get:
$$C_{p}^2-\mathrm{TCF}\nsubseteq  C_{p^{2}}^2-\mathrm{TCF}.$$
This observation can not be immediately made into a proof of the non-existence of the theory $C_{p^{\infty}}^2-\mathrm{TCF}$, for which we will need the results of Section \ref{secabsolute}.

\item As was noted in \cite[Theorem 4]{KaPie}, a model companion of the theory $T_{\infty}$ (in the notation from Proposition \ref{mcfact}) may still exist, even when the theories $(T^{\mathrm{mc}}_m)_m$ do not form an increasing chain. We come across such a situation in the case of actions of torsion Abelian groups. Namely, let us define:
$$C_{\Pp}:=\bigoplus_{p\in \Pp}C_p\cong \coli_mC_{p_1\ldots p_m},$$
where $\Pp=(p_i)_{i>0}$ is an enumeration of the set of all primes. Then, one can see (similarly as in Item $(1)$ above) that:
$$C_2-\mathrm{TCF}\nsubseteq  C_6-\mathrm{TCF},$$
but (by Theorem \ref{superthm}), the theory $C_{\Pp}$-TCF still exists.
\end{enumerate}
\end{example}
\begin{remark}
The situation from Proposition \ref{mcfact} also appears in a \emph{differential} context. In \cite{K3}, the second author gave geometric axioms of the model companion of the theory of fields with finitely many commuting Hasse-Schmidt derivations using increasing chains of theories. This last result was generalized in \cite{HK} to actions of formal groups on fields (an iterative Hasse-Schmidt derivation corresponds to an action of the formalization of the additive group).
\end{remark}

\section{Absolute Galois groups}\label{secabsolute}
In this section, we begin our proof of Theorem \ref{superthm} (the main result of this paper) by describing the absolute Galois groups of certain fields of invariants.

Firstly, we collect several notions from the theory of profinite groups, which we will often use in the sequel without any references. Proofs of these results can be found in \cite[Chapter 22.9]{FrJa}. The classical Sylow theory for finite groups generalizes smoothly to the profinite context after replacing the notion of a $p$-subgroup with the notion of a closed pro-$p$ subgroup. In particular, for $p\in \Pp$ and a profinite group $\mathcal{G}$, $p$-Sylow subgroups of $\mathcal{G}$ exist and they are conjugate. We also have the corresponding results about pronilpotent groups, that is: a profinite group $\mathcal{G}$ is pronilpotent if and only if it is the product of its unique $p$-Sylow subgroups. If $\mathcal{G}$ is a pronilpotent group and $p\in \Pp$, then we denote by $\mathcal{G}_{(p)}$ the unique $p$-Sylow subgroup of $\mathcal{G}$.  Throughout this section, ``cl'' denotes the topological closure (in an ambient profinite group).

We will need a very general result about pronilpotent groups stated below. It may be folklore, but we were unable to find it in the literature.
\begin{prop}\label{nilpinf}
Let $\mathcal{G}$ be a profinite group, $(I,\leqslant)$ be a directed partially ordered set and $(P_i)_{i\in I}$ be a direct system of closed pronilpotent subgroups of $\mathcal{G}$ (ordered by inclusion). Then, the subgroup
$$P_{\infty}:=\mathrm{cl}\left(\bigcup_{i\in I}P_i\right)$$
is pronilpotent as well.
\end{prop}
\begin{proof}
For each $i\in I$, we have
$$P_i=\prod_{p\in \Pp}(P_i)_{(p)}.$$
By \cite[Proposition 22.9.2(a)]{FrJa}, for each $p\in \Pp$ we get that $\left((P_i)_{(p)}\right)_{i\in I}$ is a direct system of pro-$p$ groups (ordered by inclusion). For $p\in \Pp$, we define:
$$P_{\infty,p}:=\mathrm{cl}\left(\bigcup_{i\in I}(P_i)_{(p)}\right).$$
Since 
%the commutator map $(x,y)\mapsto [x,y]$ is continuous, 
for each $i\in I$, we have:
$$P_i\subseteq \prod_{p\in \Pp}P_{\infty,p}$$
(it follows from $(P_i)_{(p)}\subseteq P_{\infty,p}$ for each $p\in \Pp$), we obtain that:
$$P_{\infty}\subseteq \prod_{p\in \Pp}P_{\infty,p}.$$
Therefore, it is enough to show that each $P_{\infty,p}$ is a pro-$p$ group.
% (and after showing this, we have $P_{\infty,p}=(P_{\infty})_{(p)}$).
To ease the notation, we assume that each $P_i$ is a pro-$p$ group and we aim to show that $P_{\infty}$ is a pro-$p$ group as well.

Let us fix a $p$-Sylow subgroup $P\leqslant \mathcal{G}$. For each $i\in I$, we define:
$$\mathcal{X}_i:=\left\{g\in \mathcal{G}\ |\ P_i\subseteq gPg^{-1}\right\}.$$
Since for any fixed $x\in \mathcal{G}$, the map $g\mapsto g^{-1}xg$ is continuous, it is easy to see that for each $i\in I$, the set $\mathcal{X}_i$ is closed. Since all $p$-Sylow subgroups of $\mathcal{G}$ are conjugate, each $\mathcal{X}_i$ is non-empty and $(\mathcal{X}_i)_{i\in I}$ a direct system ordered by the reversed inclusion. Since $\mathcal{G}$ is compact, we get that
$$\bigcap_{i\in I}\mathcal{X}_i\neq \emptyset.$$
Let us take $g\in \bigcap_{i\in I}\mathcal{X}_i$. Then for each $i\in I$, we have
$$P_i\subseteq gPg^{-1}.$$
Since $gPg^{-1}$ is pro-$p$, $P_{\infty}$ is also pro-$p$, which finishes the proof.
\end{proof}
We state below a crucial result about absolute Galois groups of fixed subfields of $A$-closed fields, where $A$ is a torsion Abelian group. Let us fix such  a group $A$ and we present it as:
$$A=\coli_{i\in I}A_i$$
for finite Abelian groups $A_i$ such that $A_0=\{0\}$ ($0$ appearing in the subscript is the smallest element in $(I,\leqslant)$). For each $p\in \Pp$, we denote by $A_{(p)}$ its $p$-power torsion subgroup (which can be considered as its $p$-Sylow subgroup). Let us fix an $A$-field $K$. For each $i\in I$, we denote:
$$K_i:=K^{A_i}$$
and we have the following short exact sequence
\begin{equation*}
 \xymatrix{1 \ar[r]^{}  & \gal(K) \ar[r]^{<}  &  \gal(K_i) \ar[r]^{\mathrm{res}_i\ \ \ } & \gal(K/K_i) \ar[r]^{}  & 1 ,}
\end{equation*}
where $\mathrm{res}_i$ is the appropriate restriction map.
\begin{theorem}\label{mgt}
Suppose that $K$ is $A$-closed and strict. Then we have the following.
\begin{enumerate}
\item For each $i\in I$, the profinite group $\gal(K_i)$ is pronilpotent.

\item Suppose that for each $p\in \Pp$, the group $A_{(p)}$ is finite. We enumerate the set of all primes $\Pp=(p_n)_{n>0}$ and set:
$$A_n:=A_{(p_1)}\oplus \ldots \oplus A_{(p_n)}.$$
Let us take $j,n\in \Nn$. If $j\leqslant n$, then the restriction map:
$$\mathrm{res}_n:\gal(K_n)_{(p_j)}\ra A_{(p_j)} = \left(A_n\right)_{(p_j)}$$
is a Frattini cover.
\end{enumerate}
\end{theorem}
\begin{proof}
We proceed to show Item $(1)$ and then we will notice that under the extra assumptions of Item $(2)$, the proof of Item $(1)$ gives the stronger conclusion from Item $(2)$. The following claim is crucial for our proof of Item $(1)$ and the proof of this claim is rather long. Since $K$ is a strict $A$-field, for each $i\in I$, we will identify $\gal(K/K_i)$ with $A_i$.
\\
\\
{\bf Claim}
\\
For each $i\in I$, there is a closed subgroup $W_i\leqslant \gal(K_i)$ such that:
\begin{enumerate}
\item the profinite group $W_i$ is pronilpotent;

\item we have:
$$\mathrm{res}_i(W_i)=A_i;$$

\item for each $i,j\in I$, if $i\leqslant j$, then we have:
$$\gal(K_i)\cap W_j=W_i.$$
%{\bf Fix consistent notation: ``$i\leqslant j$'' or ``$j\leqslant i$''??! Better :``$i\leqslant j$''!}
\end{enumerate}
Before proving the Claim, we will quickly see that it implies Item (1) from Theorem \ref{mgt}. Let $W$ be the common intersection of all $W_i$'s with $\gal(K)$ and $K':=(K^{\alg})^W$. By the Claim and Proposition \ref{fraexcolim}, the action of $A$ on $K$ extends to $K'$. Since $K$ is $A$-closed, we get $K=K'$. Therefore, $\gal(K)=W$ and for each $i$, we have
$$\ker(\mathrm{res}_i)=\gal(K)\subseteq W_i.$$
Since $\mathrm{res}_i(W_i)=A_i$, we get that $W_i=\gal(K_i)$, so all the profinite groups $\gal(K_i)$ are pronilpotent.
%we get that for each $i\leqslant j$, the $A_j$-field structure expands the $A_i$-field structure. Hence we get an algebraic $A$-field extension $K\subseteq %L$. Since $K$ is $A$-closed, we get that $K=L$, which implies that for each $i\in I$, we have $\gal(K_i)=W_i$, which is pronilpotent.
\begin{proof}[Proof of Claim]
For each $i\in I$ and $p\in \Pp$, let $n_{i,p}$ be the cardinality of the finite $p$-group $(A_i)_{(p)}$.
%Then each map $\mathrm{res}_{i,j}$ is a Frattini cover as well. Therefore there are tuples  $\bar{x}_{i,1},\ldots,\bar{x}_{i,k}$ of elements of $\gal(K_i)$ %such that:
We define the following set of infinite tuples:
$$\mathrm{Cl}_i:=\left\{x\in \prod_{p\in \Pp}\gal(K_i)^{n_{i,p}}\ |\ \text{$x$ satisfies the conditions (i)--(iii) below}\right\}.$$
Before stating the conditions (i)--(iii), we fix the obvious presentation of a tuple $x\in \prod_{p\in \Pp}\gal(K_i)^{n_{i,p}}$:
$$x=\left(x_{(p)}\right)_{p\in \Pp},\ \ \ \ x_{(p)}\in \gal(K_i)^{n_{i,p}}.$$
We give below the conditions defining the set $\mathrm{Cl}_i$.
\begin{enumerate}
\item[(i)] For each $p\in \Pp$, we have:
$$\mathrm{res}_i\left(x_{(p)}\right)=(A_i)_{(p)}.$$

\item[(ii)] For each $p\in \Pp$, the group $\mathrm{cl}(\langle x_{(p)}\rangle)$ is a pro-$p$ subgroup of $\gal(K_i)$.

\item[(iii)] For each $p,q\in \Pp$, if $p\neq q$ then
$$[x_{(p)},x_{(q)}]=1,$$
i.e. the coordinates of the tuple $x_{(p)}$ commute with the coordinates of the tuple $x_{(q)}$.
\end{enumerate}
For each $x\in \mathrm{Cl}_i$, we define:
$$W_i^x:=\mathrm{cl}(\langle x\rangle).$$
{\bf Subclaim 1}
\\
For each $i\in I$, we have the following.
\begin{enumerate}
\item $\mathrm{Cl}_i\neq \emptyset$.

\item $\mathrm{Cl}_i$ is a closed subset of $\prod_{p\in \Pp}\gal(K_i)^{n_{i,p}}$.

\item For each $x\in \mathrm{Cl}_i$, the profinite group $W_i^x$ is pronilpotent and $\mathrm{res}_i(W_i^x)=A_i$.
\end{enumerate}
\begin{proof}[Proof of Subclaim 1]
For the proof of Item $(1)$, we notice that by \cite[Lemma 2.8.15]{progps} there is a closed subgroup $W_i\leqslant\gal(K_i)$ such that $$\mathrm{res}_i|_{W_i}:W_i\ra A_i$$ is a Frattini cover. 
By \cite[Corollary 22.10.6(b)]{FrJa}, the universal Frattini cover $\widetilde{A_i}$ of the commutative profinite group $A_i$ is pronilpotent. Since $W_i\to A_i$ is a Frattini cover, there is a continuous epimorphism $\widetilde{A_i}\to W_i$ (see the condition (b) in \cite[Proposition 22.6.1]{FrJa}), which implies that the
profinite group $W_i$ is pronilpotent as well. We have the following decomposition:
$$W_i=\prod_{p\in \Pp}(W_i)_{(p)}.$$
For each $p\in \Pp$, we have:
$$\mathrm{res}_i\left((W_i)_{(p)}\right)=(A_i)_{(p)}.$$
Hence, there is $x_{(p)}\in (W_i)_{(p)}^{n_{i,p}}$ such that $\mathrm{res}_i(x_{(p)})=(A_i)_{(p)}$. Thus, we have:
$$x:=\left(x_{(p)}\right)_{p\in \Pp}\in \mathrm{Cl}_i$$
and $\mathrm{Cl}_i$ is non-empty.

For the proof of Item $(2)$, it is clear that the conditions (i) and (iii) from the definition of $\mathrm{Cl}_i$ are closed. Since any closed pro-$p$ subgroup of $\gal(K_i)$ is contained in a Sylow  pro-$p$ subgroup of $\gal(K_i)$, $x$ satisfies the condition (ii) for a prime $p$ if and only if the tuple $x_{(p)}$ is contained in a $p$-Sylow subgroup of $\gal(K_i)$. Hence, it is enough to check that this last condition on the tuple $x_{(p)}$ is closed. Let us fix $P$, a $p$-Sylow subgroup of $\gal(K_i)$. We set $n:=n_{i,p}$ and consider the following function:
$$\Psi:\gal(K_i)\times P^{n}\ra \gal(K_i)^{n},\ \ \ \ \Psi(g,(x_1,\ldots,x_n))=\left(gx_1g^{-1},\ldots,gx_ng^{-1}\right).$$
Since all $p$-Sylow subgroups of $\gal(K_i)$ are conjugate, the set of tuples $x_{(p)}$ satisfying our last condition coincides with the image of the function $\Psi$. Since $\Psi$ is a continuous function between compact topological spaces, its image is closed.

Item $(3)$ is obvious from the definition of the group $W_i^x$.
\end{proof}
From Item $(3)$ in Subclaim 1, we see that for each $i\in I$, there is a closed pronilpotent subgroup $W_i\leqslant \gal(K_i)$ such that $\mathrm{res}_i(W_i)=A_i$. To finish the proof of the Claim, we need to find such $W_i$'s satisfying the extra condition saying that for $i\leqslant j$, we have $W_j\cap \gal(K_i)=W_i$. Firstly, we will find $W_i$'s satisfying the following weaker condition: $W_i\subseteq W_j$ (for $i\leqslant j$).

For each $i,j\in I$ such that $i\leqslant j$, we define the following coordinate projection map:
$$\pi^j_i:\prod_{p\in \Pp}\gal(K_j)^{n_{j,p}}\ra \prod_{p\in \Pp}\gal(K_j)^{n_{i,p}},$$
where the projections are induced by the inclusions $(A_i)_{(p)}\leqslant (A_j)_{(p)}$. We define:
$$\mathrm{Cl}_i^j:=\pi^j_i(\mathrm{Cl}_j)\cap \prod_{p\in \Pp}\gal(K_i)^{n_{i,p}}.$$
{\bf Subclaim 2}
\\
For each $i,j\in I$ such that $i\leqslant j$ we have the following.
\begin{enumerate}
\item $\mathrm{Cl}_i^j\subseteq \mathrm{Cl}_i$.

\item $\mathrm{Cl}_i^j\neq \emptyset$.

\item $\mathrm{Cl}_i^j$ is closed.
\end{enumerate}
\begin{proof}[Proof of Subclaim 2]
To show Item $(1)$, we consider the following commutative diagram with exact rows:
\begin{equation*}
 \xymatrix{1 \ar[r]^{}  & \gal(K) \ar[r]^{<\  } \ar[d]^{=} &  \gal(K_i) \ar[d]^{\leqslant} \ar[r]^{\ \ \mathrm{res}_i} & A_i \ar[d]^{\leqslant} \ar[r]^{}  & 1 \\
 1 \ar[r]^{}  & \gal(K) \ar[r]^{<\ }  &  \gal(K_j) \ar[r]^{\ \ \mathrm{res}_j} & A_j \ar[r]^{}  & 1.}
\end{equation*}
The right part of this diagram is a Cartesian square, that is $\mathrm{res}_j^{-1}(A_i)=\gal(K_i)$. Therefore, each $x\in \mathrm{Cl}_i^j$ satisfies Condition (i) from the definition of the set $\mathrm{Cl}_i$. Since Conditions (ii) and (iii) are clearly satisfied for any $x\in \mathrm{Cl}_i^j$, Item $(1)$ is proved.

For Item $(2)$, it is enough to notice (by the same Cartesian square argument as above) that the condition $\mathrm{res}_j(W_j)=A_j$ implies that
$$\mathrm{res}_i\left(W_j\cap \gal(K_i)\right)=A_i.$$
Item $(3)$ is obvious, since $\pi^j_i$ is a continuous map between compact topological spaces.
\end{proof}
The set $\mathrm{Cl}_i^j$ has the following interpretation: for any $x\in \mathrm{Cl}_i^j$, there is $y\in \mathrm{Cl}_j$ such that $W_i^x\subseteq W_j^y$, so $W_i^x$ extends to $W_j^y$. We want to have this extension property ``all the way along $(I,<)$'': in particular for $i_1\leqslant i_2\leqslant i_3\leqslant \ldots$, we want to find $x_1$ such that $W_{i_1}^{x_1}$ extends to $W_{i_2}^{x_2}$ which extends to $W_{i_3}^{x_3}$, etc. To this end, for any $i_1\leqslant i_1\leqslant \ldots \leqslant i_{n}$ from $I$, we define
$$\mathrm{Cl}_{i_1}^{i_2,i_3}:=\pi^{i_2}_{i_1}\left(\mathrm{Cl}_{i_2}^{i_3}\right),\ \ \
\mathrm{Cl}_{i_1}^{i_2,\ldots,i_{n}}:=\pi^{i_2}_{i_1}\left(\mathrm{Cl}_{i_2}^{i_3,\ldots,i_{n}}\right).$$
To convey the main idea in a proper way, it is more convenient to continue in the special case when $I=\Nn$ and $\leqslant$ is the standard ordering on $\Nn$. We will point out later what one needs to do in the general case. We define:
$$\mathrm{Cl}_0^{\infty}:=\bigcap_{n=2}^{\infty}\mathrm{Cl}_0^{2,\ldots,n}.$$
As it was argued several times before in this proof, the compactness of $\gal(K_0)=\gal(K)$ implies that the set $\mathrm{Cl}_0^{\infty}$ is non-empty. Let us take $x_0\in \mathrm{Cl}_0^{\infty}$. It follows from the definition of $\mathrm{Cl}_0^{\infty}$ that there is a sequence $(x_i\in \mathrm{Cl}_i)_{i\geqslant0}$ such that for each $i\geqslant 0$ we have:
$$\pi^{i+1}_i\left(x_{i+1}\right)=x_i.$$
Let us define:
$$V_i:=W_i^{x_i}\leqslant \gal(K_i).$$
Then the profinite groups $V_i$'s are pronilpotent, they project onto the corresponding $A_i$'s and we have $V_i\subseteq V_{i+1}$, so we have achieved the first step of approximating the conditions on $W_i$ from the statement of the Claim. We will correct these $V_i$'s to satisfy the conditions from the Claim fully. This is a very general procedure, we start from the following commutative diagram of inclusions, which summarizes our situation:
\begin{equation*}
 \xymatrix{ \gal(K_0) \ar[r]^{<}  &  \gal(K_1) \ar[r]^{<} & \gal(K_2) \ar[r]^{<} &\ldots  \\
V_0 \ar[r]^{<} \ar[u]^{<} &  V_1 \ar[r]^{<} \ar[u]^{<} & V_2 \ar[r]^{<} \ar[u]^{<} & \ldots  .}
\end{equation*}
For each $0\leqslant i<j$, we define:
$$V_{i,j}:=V_j\cap \gal(K_i),\ \ \ \ V_{i}^{(1)}:=\cl\left(\bigcup_{j>i}V_{i,j}\right).$$
It is finally the right moment to use Proposition \ref{nilpinf} and thanks to this result each profinite group $V_{i}^{(1)}$ is pronilpotent. Clearly, $V_{i}^{(1)}$'s project again onto $A_i$'s and we have 
$$V_{i}\subseteq V_{i}^{(1)}\subseteq V_{i+1}^{(1)}.$$ 
For each $i\geqslant 0,n>0$, we define now:
$$V_i^{(n+1)}:=\left(V_i^{(n)}\right)^{(1)},\ \ \ \ V_i^{(\omega)}:=\cl\left(\bigcup_{n=1}^{\infty}V_i^{(n)}\right).$$
Again, $V_i^{(\omega)}$'s are pronilpotent, they project onto $A_i$'s and we have $V_i^{(\omega)}\subseteq V_{i+1}^{(\omega)}$. We can continue like this using transfinite induction as long as we wish. However, this procedure must finish after some (ordinal) number of steps --- it is possible that countably many steps are enough, but for sure it is enough to take 
$$\kappa:=\left|\mathrm{Aut}\left(K^{\alg}\right)\right|^+$$ 
of them. Then, for each $i>0$ we define:
$$W_i:=V_i^{(\kappa)}$$
and, by the construction, these $W_i$'s satisfy the conditions of the Claim, which finishes the proof of the Claim in the case of $(I,\leqslant)=(\Nn,\leqslant)$.

We sketch now how one can proceed in the case of an arbitrary directed poset $(I,\leqslant)$. We choose a maximal antichain $\mathcal{A}$ in $I$. Without loss of generality, $\mathcal{A}$ is infinite (otherwise, $I$ can be taken to be $\Nn$). Then, we can assume that:
$$(I,\leqslant)=\left([\mathcal{A}]^{<\omega},\subseteq\right),$$
which is the set of finite subsets of $\mathcal{A}$ ordered by the inclusion relation. For any $n>0$, let $I_n$ denote the subset of $I$ consisting of subsets of $\mathcal{A}$ of cardinality $n$. Then we have:
$$I=I_1\cupdot I_2\cupdot \ldots$$
and we can repeat the previous argument with taking an extra care about all the elements of $I_n$ at each level $n$. Namely, for any $i\in I$ we define
$$\mathrm{Cl}_i^1:=\bigcap_{a\in \mathcal{A}}\mathrm{Cl}_{i}^{ia},$$
where $ia:=i\cup \{a\}$.
The set $\mathrm{Cl}_i^1$ is non-empty by the arguments as above. Then, we define
$$\mathrm{Cl}_{i}^{2}:=\pi^{ia}_{i}\left(\mathrm{Cl}^{1}_{ia}\right),\ \ \
\mathrm{Cl}_{i}^{n+1}:=\pi^{i_2}_{i_1}\left(\mathrm{Cl}^{n}_{ia}\right)$$
and we can continue as in the case of $I=\Nn$ above.

Hence, we have obtained the subgroups $W_i\leqslant \gal(K_i)$ for each $i\in I$ which satisfy Items $(1)$ and $(2)$ from the Claim and such that for each $i,j\in I$, if $i\leqslant j$, then $W_i\subseteq W_j$. To correct these $W_i$'s to satisfy Item $(3)$ from Claim, one can just repeat the procedure described in the case of $I=\Nn$.
\end{proof}
As it was explained immediately after the statement of the Claim, Item $(1)$ (from the statement of Theorem \ref{mgt}, which we are still proving) directly follows from the Claim, whose proof was just finished above.

We proceed now towards the proof of Item $(2)$. Having the extra assumptions from Item $(2)$, we:
\begin{itemize}
\item set $r_{n,p}$ as the \emph{rank} (rather than just the cardinality) of $(A_n)_{(p)}$;

\item replace Condition (i) from the proof of the Claim with the following condition:
\begin{enumerate}
  \item[(i')] $\left\langle \mathrm{res}_n\left(x_{(p)}\right)\right\rangle =(A_n)_{(p)}$.
\end{enumerate}
\end{itemize}
Condition (i') is still closed, since there is a fixed finite set of sequences of length $r_{n,p}$ generating the group $(A_n)_{(p)}$.

We consider now the following commutative diagram:
\begin{equation*}
 \xymatrix{\gal(K_j)_{(p_j)} \ar[rr]^{\ \ \mathrm{res}_j} & & (A_j)_{(p_j)}\ar[ddd]^{=}  \\
 (V_j)_{(p_j)} \ar[rru]^{} \ar[u]^{\leqslant} \ar[d]^{\leqslant} & &  \\
 (V_{j+1})_{(p_j)} \ar[rrd]^{\mathrm{F.c.}} \ar[d]^{\leqslant} &  & \\
 \gal(K_{j+1})_{(p_j)} \ar[rr]^{\ \ \mathrm{res}_{j+1}} & & (A_{j+1})_{(p_j)}.}
\end{equation*}
By Remark \ref{unfortun}(2), any continuous epimorphism of pro-$p$ groups of the same finite rank is necessarily a Frattini cover. Hence, the map $\mathrm{res}_{j+1}$ restricted to $(V_{j+1})_{(p_j)}$ is a Frattini cover as indicated in the diagram above. Since 
$$\mathrm{res}_{j+1}\left((V_j)_{(p_j)}\right) = \mathrm{res}_j\left((V_j)_{(p_j)}\right) = (A_j)_{(p_j)} = (A_{j+1})_{(p_j)},$$
 we get that $(V_{j+1})_{(p_j)}=(V_{j})_{(p_j)}$. Hence, for each $n\geqslant j$, we have $(V_n)_{(p_j)}=(V_j)_{(p_j)}$.

Therefore, if we repeat the process of getting $W_n$'s from $V_n$'s appearing at the end of the proof of Item $(1)$, then for each $n\geqslant j$, we have $$(W_n)_{(p_j)}=(V_n)_{(p_j)}.$$
Since for each $n\geqslant j$, $(V_n)_{(p_j)}$ is a Frattini cover of $(A_n)_{(p_j)}$ and $\gal(K_n)=W_n$, the proof is finished.
\end{proof}
\begin{remark}
Item $(1)$ in Theorem \ref{mgt} cannot be improved towards the conclusion from the statement of Item $(2)$. To see that, let us consider an existentially closed $C_{p^{\infty}}^2$-field $K$ and let $A_n:=C_{p^{n}}^2$. By the Claim from the beginning of Section \ref{secneg}, the restriction map:
$$\gal(K_n)\ra A_n=C_{p^{n}}^2$$
is not a Frattini cover.
\end{remark}

\section{Negative results}\label{secneg}
This section is mostly about the proof of the left-to-right implication from Theorem \ref{superthm}. Using Remark \ref{equivcond}(2), we assume that there is an infinite strictly increasing sequence
$$P_1=C_p^2<P_2<P_3<\ldots$$
such that each $P_i$ is a finite $p$-subgroup of $A$. We aim to show that the theory $A$-TCF does not exist.
\\
\\
Assume that the theory $A$-TCF exists and let $K$ be an $|A|^+$-saturated and existentially closed $A$-field. We will reach a contradiction.

By Theorem \ref{mgt}(1), for each $i\in I$, the profinite group $\gal(K_i)$ is pronilpotent, hence it decomposes as:
$$\gal(K_i)=\prod_{p\in \Pp}\gal(K_i)_{(p)}.$$
By Corollary \ref{constpac}, for each $i\in I$, the field $K_i$ is PAC. Hence (see \cite[Theorem 11.6.2]{FrJa}), the profinite group $\gal(K_i)$ is projective and each pro-$p$ group $\gal(K_i)_{(p)}$ is projective. As in Example \ref{sectworem}(1), $\gal(K_i)_{(p)}$ is pro-$p$ free, so there is a cardinal $\kappa_i$ such that:
$$\gal(K_i)_{(p)}\cong \widehat{F}_{\kappa_i}(p).$$
{\bf Claim}
\\
For each $i\in I$ such that $C_p^2=P_1\subseteq A_i$, the restriction map
$$\gal(K_i)_{(p)}\ra (A_i)_{(p)}$$
is not a Frattini cover.
\begin{proof}[Proof of Claim]
Assume not and let us take $i\in I$ as above such that the map
$$\mathrm{res}_i:\gal(K_i)_{(p)}\ra (A_i)_{(p)}$$
is a Frattini cover. Let $r$ be the rank of $(A_i)_{(p)}$. By our assumption, $r$ is finite and $r\geqslant 2$. Since the map $\mathrm{res}_i$ is a Frattini cover, we get that $r=\mathrm{rk}(\gal(K_i)_{(p)})$, hence (see \cite[Corollary 22.5.3]{FrJa}) $r=\kappa_i$.

For any $j\in I$ such that $i\leqslant j$, we have the following commutative diagram:
\begin{equation*}
 \xymatrix{1 \ar[r]^{}  &  \gal(K)_{(p)}  \ar[d]^{=}  \ar[r]^{<\ \ \ \ \ \ \ } & \gal(K_i)_{(p)}\cong \widehat{F}_{r}(p) \ar[r]^{\ \ \ \ \ \ \ \mathrm{res}_i} \ar[d]^{<} & (A_i)_{(p)} \ar[d]^{<} \ar[r]^{} & 1 \\
           1 \ar[r]^{}  &  \gal(K)_{(p)}    \ar[r]^{<\ \ \ \ \ \ \ } & \gal(K_j)_{(p)}\cong \widehat{F}_{\kappa_j}(p) \ar[r]^{\ \ \ \ \ \ \ \mathrm{res}_j}  & (A_j)_{(p)}  \ar[r]^{} & 1.}
\end{equation*}
Hence, we have:
$$t_j:=\left[\widehat{F}_{\kappa_j}(p):\widehat{F}_{r}(p)\right]=\left[(A_j)_{(p)}:(A_i)_{(p)}\right].$$
Therefore, by our main assumption on the sequence of groups $(A_i)_i$, the indices $t_j$ go to infinity when $j\rightarrow \infty$. This leads to a contradiction by the profinite version of Nielsen-Schreier formula (\cite[Proposition 17.6.2]{FrJa}), which we observe below.

Since $r$ is finite, $\kappa_j$ is finite as well for each $j\geqslant i$ and we have
$$r=1+t_j(\kappa_j-1).$$
Since $r\geqslant 2$, then for each $j\geqslant i$ we have $\kappa_j-1>0$. Since $t_j$'s go to infinity with $j$, then the constant $r$ tends to infinity as well, which is obviously a contradiction.
\end{proof}
Using the Claim above, we obtain that for each $i\in I$ such that $C_p^2\subseteq A_i$, the map:
$$\mathrm{res}_i:\gal(K_i)_{(p)}\ra \left(A_i\right)_{(p)}$$
is not a Frattini cover, which is witnessed by a closed proper subgroup $H_i<\gal(K_i)_{(p)}$ such that:
$$\mathrm{res}_i(H_i)=\left(A_i\right)_{(p)}.$$
Without loss of generality, $H_i$ is a maximal proper closed subgroup of $\gal(K_i)_{(p)}$. Since the profinite group $\gal(K_i)_{(p)}$ is pro-$p$, we get that
$$\left[\gal(K_i)_{(p)}:H_i\right]=p.$$
Hence, by Lemma \ref{fraextbest} (applied to $\mathcal{G}_0:=H_i\times \prod_{q\neq p}\gal(K_i)_{(q)}$), for each such $i\in I$, there is an $A_i$-field extension $K\subset L_i$ of degree $p$. Since $K$ is $|A|^+$-saturated,
 there is an $A$-field extension $K\subset L$ of degree $p$, which gives our final contradiction ($K$ is an existentially closed $A$-field, so it is also $A$-closed) and finishes the proof of the left-to-right implication in Theorem \ref{superthm}.
%\\
%\\
%Let us isolate now one interesting conclusion...
%\begin{remark}\label{galnotp}
%Let $t\in \Nn_{\geqslant 2}$ and $K$ be a $C_p^t$-field which is $C_p^t$-closed and strict. Then the action of $C_p^t$ on $K$ can not be extended to a strict %$C_{p^2}^t$-action.
%\\
%\textbf{Is \emph{strict} necessary?}
%\end{remark}
%\begin{proof}
%Suppose $K$ is a $C_p^t$-field which is $C_p^t$-closed and strict. Let $C:=\mathrm{Fix}(C_p^t)$. We identify $\gal(K/C)$ with $C_{p}^{t}$. We have the %following exact sequence:
%\begin{equation*}
%\xymatrix{1  \ar[r]^{}  & \gal(K) \ar[r]^{<\ }  & \gal(C) \ar[r]^{\ \ \mathrm{res}} & C_{p}^{t} \ar[r]^{} & 1.}
%\end{equation*}
%Since $K$ is is $C_p^t$-closed, the restriction map $\mathrm{res}$ is the universal Frattini cover. Using the profinite Nielsen-Schreier formula again we %get:
%$$\gal(C)\cong \widehat{F}_p(t),\ \ \ \ \gal(K)\cong \widehat{F}_p(k),$$
%which leads to a contradiction similarly as ...
%\\
%{\bf actually, isn't it EXACTLY the same proof??!}
%\end{proof}

\section{Positive results}\label{secpos}

This section is about the proof of the right-to-left implication in Theorem \ref{superthm}. Similarly as in Section \ref{secabsolute}, we start this section with a very general result, which will be needed later.
\begin{prop}\label{pacint}
Let $K_1\supseteq K_2 \supseteq K_3 \supseteq \ldots$ be a decreasing tower of fields and let
$$K_{\infty}:=\bigcap_{n=1}^{\infty}K_n.$$
We assume that $K_{\infty}$ is PAC and that the field extension $K_{\infty}\subseteq K_1$ is algebraic. Let $V$ be an algebraic variety over $K_{\infty}$ such that for all $n>0$, we have $V(K_n)\neq \emptyset$. Then $V(K_{\infty})\neq \emptyset$.
%{\bf Check Remark \ref{pacintrem}(1). Then: some discussions what does ``constructible'' actually mean?! Just a quantifier free field formula?}
\end{prop}
\begin{proof}
Let $V=V_1\cup \ldots \cup V_d$ be the decomposition of $V$ into irreducible components over $K_{\infty}^{\alg}$. We assume that the result does not hold (for this fixed tower $(K_n)_n$) and take a counterexample $V$, which is minimal with respect to $(\dim(V),d)$. We will reach a contradiction.

Since $K_{\infty}$ is PAC, we have $d>1$. Without loss of generality, we may assume that for each $n$, the set
$$Q_n:=V(K_n)\cap V_1\left(K_{\infty}^{\alg}\right)$$
is non-empty. Let $W_n$ be the Zariski closure of $Q_n$ for each $n$ inside $V_1\left(K_{\infty}^{\alg}\right)$. Then, for each $n$ we have the following:
\begin{itemize}
  \item $W_n$ is non-empty;
  \item $W_n$ is defined over $K_n$;
  \item  $W_n\subseteq V_1$;
  \item $W_n(K_n)\neq \emptyset$;
  \item $W_n$'s form a descending chain.
%  {\bf Is it OK? Some problems during seminar requiring perfectness of $K_i$ and extensions $K_{i+1}\subseteq K_i$ being algebraic.}
\end{itemize}
Hence, there is a variety $W$ such that for $n\gg 0$, we have $W=W_n$. Therefore, $W$ is defined over $K_{\infty}=\bigcap K_n$, and $W$ satisfies the assumption of the statement we are proving. But, since $W\subseteq V_1$ and $d>1$, either $\dim(W)<\dim(V)$ or the number of irreducible components (over $K_{\infty}^{\alg}$) of $W$ is smaller than $d$ (actually, if $\dim(W)=\dim(V)$, then $W=V_1$, so $W$ is absolutely irreducible). By minimality of $V$, we get $W(K_{\infty})\neq \emptyset$. But then $V(K_{\infty})\neq \emptyset$, a contradiction.
\end{proof}
\begin{remark}\label{pacintrem}
\begin{enumerate}
\item The conclusion of Proposition \ref{pacint} can be easily strengthened by replacing the variety $V$ with a \emph{constructible set}. By a constructible set defined over a field $M$, we just mean a quantifier-free formula in the language of rings with parameters from $M$. If we evaluate this formula on $M^{\alg}$, then we get a ``classical'' constructible set which is a Noetherian topological space with the induced Zariski topology and its irreducible components are constructible sets as well (defined over $M^{\alg}$). To make the proof of  Proposition \ref{pacint} work in this context, one only needs to notice that since for an absolutely irreducible variety $W$ over a PAC field $C$, we have that $W(C)$ is Zariski dense in $W$ (\cite[Proposition 11.1.1]{FrJa}), then any absolutely irreducible constructible set over $C$ has a $C$-rational point as well.

\item If we do not put assumptions on the intersection of the tower of fields, then Proposition \ref{pacint} fails, we give an example below. For $n>0$, let
$$K_n:=\left(\Ff_3^{\alg}\right)^{\Zz_{p_1}\times \ldots\times \Zz_{p_n}}$$
(we enumerate the primes $\Pp=(p_n)_{n>0}$). Each field extension $\Ff_3\subset K_n$ is infinite algebraic, so, by \cite[Corollary 11.2.4]{FrJa}, each field $K_n$ is PAC. Therefore, it is enough to take an absolutely irreducible variety $V$ over $\Ff_3$ such that $V(\Ff_3)=\emptyset$, for example:
$$V=V(Y^2-X^3+X+1)$$
(we could have taken any finite field in place of $\Ff_3$, see \cite[Example 11.2.9]{FrJa}).
\end{enumerate}
\end{remark}
We proceed towards the right-to-left implication in Theorem \ref{superthm}. Let us fix first a ``good'' torsion Abelian group $B$ of a special kind, that is we assume that for all $p\in \Pp$, the $p$-primary part subgroup $B_{(p)}$ is finite. We also fix an enumeration of the primes $\Pp=(p_n)_{n>0}$ and for each $n\in \Nn$, we define the finite subgroup of $B$:
$$B_n:=B_{(p_1)}\oplus \ldots \oplus B_{(p_n)}.$$
Then $(B_n)_n$ is an increasing sequence such that $B=\bigcup_n B_n$ as in the assumptions of Theorem \ref{mgt}(2).
Let $K$ be a $B$-field and, as usual, we define $K_n$ as $K^{B_n}$. The main point is to show the following result below. We would like to point out that the conditions $(1)$--$(3)$ below are exactly the \emph{Galois axioms} from Definition \ref{galaxioms} (see Remark \ref{galoisok}).
\begin{theorem}\label{axioms}
The $B$-field $K$ is existentially closed if and only if the following conditions hold:
\begin{enumerate}
\item $K$ is strict and perfect;

\item for each $n\in \Nn$, $K_n$ is PAC;

\item we have
$$\gal(K)\cong \ker\left(\widetilde{\widehat{B}}\ra \widehat{B}\right)=\prod_{t>0}\ker\left(\widetilde{B_{(p_t)}} \ra B_{(p_t)}\right),$$
and for each $n>0$, we have
$$\gal(K_n) \cong \widetilde{B_n}\times \prod_{t>n}\ker\left(\widetilde{B_{(p_t)}} \ra B_{(p_t)}\right).$$
\end{enumerate}
\end{theorem}
\begin{proof} For the implication ``$\Rightarrow$'', we notice first that any e.c. $B$-field is strict and perfect (see e.g. \cite[Lemma 3.1]{HK3} and \cite[Lemma 3.4]{HK3}, which hold for an arbitrary group), so we get Item $(1)$. Item $(2)$ follows from Corollary \ref{constpac}. We proceed to show Item $(3)$. Let us fix $n\in \Nn$. By Theorem \ref{mgt}(1), the profinite group $\gal(K_n)$ is pronilpotent, hence we have:
$$\gal(K_n)=\prod_{p\in \Pp}\gal(K_n)_{(p)}.$$
We need to show that for any $j>0$, we have:
\begin{enumerate}
\item[(i)] if $j\leqslant n$, then
$$\gal(K_n)_{(p_j)}\cong \widetilde{B_{(p_j)}};$$

\item[(ii)] if $j> n$, then
$$\gal(K_n)_{(p_j)}\cong \ker\left(\widetilde{B_{(p_j)}}\ra B_{(p_j)}\right).$$
\end{enumerate}
Since $K_n$ is PAC, we get that $\gal(K_n)$ is projective as well as $\gal(K_n)_{(p)}$ for each $p\in \Pp$. In the situation of Item (i), we get what we want directly from Theorem \ref{mgt}(2).

In the situation of Item (ii), we consider the following short exact sequence:
\begin{equation*}
 \xymatrix{1 \ar[r]^{}  &  \gal(K_n)  \ar[r]^{} & \gal(K_j)\ar[r]^{\mathrm{res}\ \ \ \ }  & \gal(K_n/K_j) \ar[r]^{} & 1.}
\end{equation*}
Since $K_j=K^{B_j}$ and $K_n=K^{B_n}$, we get (using that $K$ is a strict $B$-field) the following:
$$\gal(K_n/K_j)\cong B_j/B_n\cong B_{(p_{n+1})}\oplus \ldots \oplus B_{(p_j)}.$$
By the isomorphism above and Theorem \ref{mgt}(2), we get the following  short exact sequence:
\begin{equation*}
 \xymatrix{1 \ar[r]^{}  &  \gal(K_n)_{(p_j)}  \ar[r]^{} & \gal(K_j)_{(p_j)}\cong \widetilde{B_{(p_j)}} \ar[r]^{\ \ \ \ \ \ \ \ \mathrm{res}}  & B_{(p_j)} \ar[r]^{} & 1,}
\end{equation*}
which gives the desired description of $\gal(K_n)_{(p_j)}$.
\\
\\
For the implication ``$\Leftarrow$'', let us assume that $K$ is a $B$-field satisfying the conditions $(1)$--$(3)$ above. We need the following conclusion of the Galois axioms in this case.
\\
\\
{\bf Claim 1}
\\
$K$ is $B$-closed.
\begin{proof}[Proof of Claim 1]
Let $K\subseteq K'$ be an algebraic $B$-field extension. We aim to show that $K'=K$. For each $n>0$, $K\subseteq K'$ is an algebraic $B_n$-field extension. Let $K_n':=(K')^{B_n}$ and $\mathcal{G}_n\leqslant \gal(K_n)$ be a closed subgroup such that
$$K_n'=\left(K^{\alg}\right)^{\mathcal{G}_n}.$$
Let
$$\mathrm{res}_n:\gal(K_n)\ra B_n=\gal(K/K_n)$$
be the restriction map. By Lemma \ref{fraextbest}, we obtain that
$$\mathrm{res}_n(\mathcal{G}_n)=B_n.$$
By our assumption, the profinite group $\gal(K_n)$ is pronilpotent and for each $t\leqslant n$ the map
$$\mathrm{res}_n:\gal(K_n)_{(p_t)}\ra B_{(p_t)}=\gal(K/K_n)_{(p_t)}$$
is a (necessarily universal) Frattini cover. Hence, for each $t\leqslant n$, we obtain:
$$(\mathcal{G}_n)_{(p_t)}=\widetilde{B_{(p_t)}}=\gal(K_n)_{(p_t)}.$$
In particular, for any $n>0$ we obtain (by taking $t=n$):
\begin{equation}
(\mathcal{G}_n\cap \gal(K))_{(p_n)}=\widetilde{B_{(p_n)}}\cap \gal(K)=\ker\left(\widetilde{B_{(p_n)}} \ra B_{(p_n)}\right)=\gal(K)_{(p_n)}.\tag{$*$}
\end{equation}
However, for each $n>0$, we have (see Lemma \ref{fraextbest}):
$$\left(K^{\alg}\right)^{\mathcal{G}_n\cap \gal(K)}=K'.$$
Hence, there is a closed subgroup $\mathcal{H}\leqslant \gal(K)$ such that for every $n>0$ we have:
$$\mathcal{H}=\mathcal{G}_n\cap \gal(K),\ \ \ \ K'=\left(K^{\alg}\right)^{\mathcal{H}}.$$
By $(*)$, we get that for each $n>0$:
$$\mathcal{H}_{(p_n)}=\gal(K)_{(p_n)}.$$
Therefore, $\mathcal{H}=\gal(K)$ and $K=K'$, which we needed to show.
\end{proof}
The next claim is just a restatement of the Galois axioms.
\\
\\
{\bf Claim 2}
\\
For any $n>0$, we have the following commutative diagram with exact rows (where, for clarity, we skip the trivial groups):
\begin{equation*}
 \xymatrix{\prod_{t>0}\ker\left(\widetilde{B_{(p_t)}} \to B_{(p_t)}\right)  \ar[r]^{<\ \ \ } &
 \widetilde{B_n}\times \prod_{t>n}\ker\left(\widetilde{B_{(p_t)}} \to B_{(p_t)}\right) \ar[r]^{}   & B_{n}   \\
        \gal(K)  \ar[u]_{\cong}  \ar[r]^{<\ \ \ } & \gal(K_{n}) \ar[r]^{} \ar[u]_{\cong} & \gal(K/K_n) \ar[u]_{\cong} .}
\end{equation*}
From now on, we identify all the isomorphic objects appearing in Claim 2. For $n>0$, we define:
$$K_{(n)}:=\left(K^{\alg}\right)^{\widetilde{B_n}}.$$
Since we have:
$$\widetilde{B_n} \cdot  \prod_{t>0}\ker\left(\widetilde{B_{(p_t)}}\to B_{(p_t)}\right)=
\widetilde{B_n}\times \prod_{t>n}\ker\left(\widetilde{B_{(p_t)}} \to B_{(p_t)}\right),$$
we get by Claim 2 that:
$$\gal\left(K_{(n)}\right) \cdot \gal(K) =\gal(K_n),$$
hence we obtain:
$$K_{(n)}\cap K=K_n.$$
Since $K_n\subseteq K$ is a finite Galois extension, we obtain that $K_{(n)}$ is linearly disjoint from $K$ over $K_n$ (see \cite[Corollary 2.5.2]{FrJa} and the discussion below its proof). We define now:
$$K_{(n)}':=K_{(n)}K\cong K_{(n)}\otimes_{K_n}K.$$
From the isomorphism above, $K_{(n)}'$ is naturally a $B_n$-field extension of $K$ and we also note that $K_{(n)}'$ is strict and perfect.
\\
\\
{\bf Claim 3}
\\
The fields $K_{(n)}'$ form a decreasing tower and we have the following:
$$\bigcap_{n=1}^{\infty}K_{(n)}'=K.$$
\begin{proof}[Proof of Claim 3]
From the definition of the field $K_{(n)}'$ and Claim 2, we obtain that:
$$\gal\left(K_{(n)}'\right)=\gal(K)\cap \gal(K_{(n)})=\ker\left(\widetilde{B_n}\to B_n\right).$$
Hence, we get (``cl'' below denotes the topological closure inside the profinite group $\gal(K)$):
\begin{IEEEeqnarray*}{rCl}
\mathrm{cl}\left(\bigcup_{n=1}^{\infty} \gal\left(K_{(n)}'\right)\right) & = & \mathrm{cl}\left(\bigcup_{n=1}^{\infty} \ker\left(\widetilde{B_n}\to B_n\right)\right) \\
 &= & \ker\left( \widetilde{\widehat{B}} \to \widehat{B} \right)\\
 &= & \gal(K),
\end{IEEEeqnarray*}
which yields the claim by the Galois theory.
\end{proof}
By Claim 3, we see that the fields $K_{(n)}'$ approximate our field $K$. The next claim says that they also ``logically approximate'' $K$, in the sense that these fields have better and better model-theoretic properties.
\\
\\
{\bf Claim 4}
\\
For each $n>0$, we have:
$$K_{(n)}'\models B_n-\mathrm{TCF}.$$
\begin{proof}[Proof of Claim 4]
From the definition of the $B_n$-field $K_{(n)}'$, it follows that:
$$\left(K_{(n)}'\right)^{B_n}=K_{(n)}.$$
Since $K_n\subseteq K_{(n)}$ is an algebraic field extension and $K_n$ is PAC, we get that $K_{(n)}$ is PAC as well. By the definition of $K_{(n)}$, we get that
$$\gal\left(K_{(n)}\right)\cong \widetilde{B_n},$$
hence, by Proposition \ref{finitecr}, we obtain that the $B_n$-field $K_{(n)}'$ is existentially closed.
\end{proof}
We are ready to show that $K$ is an existentially closed $B$-field. Let us take a quantifier-free $L_{B}$-formula $\varphi(x)$ over $K$ and a $B$-field extension $K\subseteq K'$ such that:
$$K'\models \exists x\varphi(x).$$
We aim to show that $K\models \exists x\varphi(x)$.

Let $N>0$ be such that $\varphi(x)\in L_{B_N}$. Since $K$ is $B$-closed (Claim 1), the field extension $K\subseteq K'$ is regular. Let us take an arbitrary $n\geqslant N$. Since the field extension $K\subseteq K'_{(n)}$ is algebraic, $K'$ is linearly disjoint from $K'_{(n)}$ over $K$ (by the definition of regular extensions). Therefore, we have
$$K'K'_{(n)}\cong K'\otimes_KK'_{(n)}$$
and the field $K'K'_{(n)}$ has a natural $B_n$-field structure extending those on $K'$ and $K'_{(n)}$ over $K$. Since the formula $\varphi(x)$ is quantifier-free, we have $K'K'_{(n)}\models \exists x\varphi(x)$. Since $K_{(n)}'$ is an existentially closed $B_n$-field (Claim 4), we have $K_{(n)}'\models \exists x\varphi(x)$.

For each $n\geqslant N$, the $B_N$-field $K'_{(n)}$ is bi-interpretable with the pure field $(K'_{(n)})^{B_N}$ (see \cite[Remark 2.3]{HK3}). To proceed, we need the following claim. The notion of ``uniform bi-interpretability'' from this claim will be explained in the beginning of its proof.
% and this explanation will also make sure that the (crucial!) moreover part of the claim holds as well.
In this claim, we also set $K_{(\infty)}':=K$ and $K_{(\infty)}:=K_N$.
\\
\\
{\bf Claim 5}
\\
The bi-interpretability between the $B_N$-field $K'_{(n)}$ and the pure field $(K'_{(n)})^{B_N}$ is uniform with respect to $n\in \{N,N+1,\ldots,\infty\}$. In particular, there is a quantifier-free formula $\psi(y)$ in the language of fields with parameters from $K_N$ such that for all $n\in \{N,N+1,\ldots,\infty\}$ we have:
$$K_{(n)}'\models \exists x\varphi(x)\ \ \ \ \Leftrightarrow\ \ \ \ \left(K_{(n)}'\right)^{B_N}\models \exists y\psi(y).$$
\begin{proof}[Proof of Claim 5]
If $G$ is a finite group of order $e$, $F$ is a $G$-field, and $M:=F^G$, then (see \cite[Remark 2.3]{HK3}) there are $M$-bilinear maps
$$m,a:M^e\times M^e\ra M^e$$
such that $(M^e,m,a)$ is naturally bi-interpretable with the field $F$. Similarly, in the case of the $G$-action, there are $M$-linear maps
$$g_1,\ldots,g_e:M^e\ra M^e$$
such that $(M^e,m,a,g_1,\ldots,g_e)$ is bi-interpretable with the $G$-field $F$.

To prove our claim, it is enough to show that there are fixed $K_N$-bilinear maps, which give $K'_{(n)}$ the $B_N$-field structure for each $n\in \{N,N+1,\ldots,\infty\}$ (we use the above observation for a fixed $G=B_N$ and where $F=K'_{(n)}$ and $M=(K'_{(n)})^{B_N}$ vary with $n$). To this end, it is enough to show that for all $n\in \{N,N+1,\ldots,\infty\}$ we have:
\begin{equation}
K'_{(n)}\cong \left(K'_{(n)}\right)^{B_N}\otimes_{K_N}K.\tag{$\dagger$}
\end{equation}
We have the following commutative diagram\footnote{We thank Junguk Lee for drawing a version of this diagram for us.} of field extensions, where the arrows are the inclusions and the Galois groups are indicated over some of the arrows:
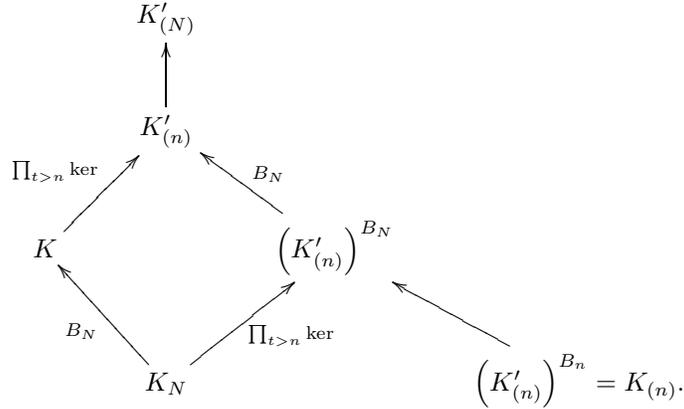
\begin{equation*}
 \xymatrix{ & K_{(N)}' & &    \\
 & K_{(n)}' \ar[u]^{} & &   \\
K \ar[ru]^{\prod_{t> n}\ker} &  & \left(K_{(n)}'\right)^{B_N} \ar[lu]_{B_N} &  \\
& K_N \ar[ru]_{\prod_{t> n}\ker} \ar[lu]^{B_N} . }
%&  & \left(K_{(n)}'\right)^{B_n}=K_{(n)} \ar[lu]^{}
\end{equation*}
Thanks to the diagram above, we see that:
\begin{itemize}
\item the field $K'_{(n)}$ is the compositum of the fields $\left(K'_{(n)}\right)^{B_N}$ and $K$;

\item the fields $\left(K'_{(n)}\right)^{B_N}$ and $K$ are linearly disjoint over $K_N$.
\end{itemize}
Hence, we get the isomorphism from $(\dagger)$ above.
\end{proof}
Let us take the quantifier-free formula $\psi(y)$ in the language of fields from Claim 5. This formula corresponds to (or even: ``this formula \emph{is}'', see Remark \ref{pacintrem}(1)) a constructible set $V$ defined over $K_N$. By Claim 5, it is enough to show that $V(K_N)\neq \emptyset$. By Claim 5 again, we get that for each $n\geqslant N$, we have 
$$V\left(\left(K'_{(n)}\right)^{B_N}\right)\neq \emptyset.$$ 
By Claim 3, we obtain that:
$$\bigcap_{n=N}^{\infty}\left(K'_{(n)}\right)^{B_N}=K_N.$$
Therefore, Remark \ref{pacintrem}(1) implies that $V(K_N)\neq \emptyset$, which finishes the proof thanks to Claim 5. (One could also arrange the original formula $\varphi(x)$ in such a way that the resulting formula $\psi(x)$ defines a variety. Then, using Proposition \ref{pacint} (rather than Remark \ref{pacintrem}(1)) would be enough.)
\end{proof}
\begin{remark}\label{galoisok}
As noted in the Introduction (below Definition \ref{galaxioms}), to see that the Galois axioms from Theorem \ref{axioms} are first-order, it is enough to show that all the absolute Galois groups $\gal(K_i)$ appearing there are small. It is clear that if the profinite group $\mathcal{G}$ is the product of its $p$-Sylow subgroups $\mathcal{G}_{(p)}$, then $\mathcal{G}$ is small if and only if each $\mathcal{G}_{(p)}$ is small. By Theorem \ref{axioms}, the profinite groups $\gal(K_i)_{(p)}$ are small, because they are topologically finitely generated.

Therefore, for a torsion Abelian group $B$ such that for all $p\in \Pp$, the $p$-power torsion subgroup $B_{(p)}$ is finite, we get that the theory $B$-TCF exists and it is axiomatised by Galois axioms from the statement of Theorem \ref{axioms}.
\end{remark}
We can conclude now the proof of our main result.
\begin{proof}[Proof of Theorem \ref{superthm}]
Since the left-to-right implication was proved in Section \ref{secneg}, it is enough to show the right-to-left implication and the moreover part of Theorem \ref{superthm}.

For the right-to-left implication, let us assume that for each prime $p$, the $p$-primary part of $A$ is either finite or it is the Pr\"{u}fer $p$-group. We decompose $A$ as:
$$A=A_f\oplus A_{\infty},$$
where for each $p\in \Pp$, we have that $(A_f)_{(p)}$ is finite, and $(A_{\infty})_{(p)}=C_{p^{\infty}}$ or $(A_{\infty})_{(p)}$ is trivial. Let us set:
$$P_{\infty}:=\{p\in \Pp\ |\ (A_{\infty})_{(p)}=C_{p^{\infty}}\}=(p_i)_{i>0},$$
$$P_{0}:=\{q\in \Pp\ |\ (A_f)_{(q)}\neq 0\}=(q_i)_{i>0},$$
and for any $m\in \Nn$ and $n>0$ we define:
$$A^{(m)}:=A_f\oplus\bigoplus_{p\in P_{\infty}}C_{p^{m}},\ \ \ \ \     \left(A_f\right)_n:=\bigoplus_{k=1}^n\left(A_f\right)_{(q_k)},\ \ \ \ \ \left(A^{(m)}\right)_n:=\left(A_f\right)_n\oplus\bigoplus_{k=1}^nC_{p_k^{m}}.$$
Then $A$ is the increasing union of the subgroups $A^{(m)}$ and each $A^{(m)}$ satisfies the assumptions on the group $B$ in the statement of Theorem \ref{axioms}. By Theorem \ref{axioms} and Remark \ref{galoisok}, for each $m\in \Nn$ the theory $A^{(m)}$-TCF exists and it is axiomatized by the Galois axioms. By Fact \ref{inter1} and Proposition \ref{mcfact}, it is enough to show that if $K\models A^{(m+1)}-\mathrm{TCF}$, then $K|_{L_{A^{(m)}}}\models A^{(m)}-\mathrm{TCF}$. The proof of this last assertion does not differ  much from the proof appearing in Example \ref{cpinftyex}. Let us take $n\in \Nn$ and let us set $P_n:=p_1\ldots p_n$. Then, we have:
$$\left(A^{(m)}\right)_n=P_n\left(A^{(m+1)}\right)_n.$$
Let us also denote:
$$K_n:=K^{\left(A^{(m+1)}\right)_n},\ \ \ \ \ K_n':=K^{\left(A^{(m)}\right)_n}.$$
By Theorem \ref{axioms}(3), we get:
\begin{IEEEeqnarray*}{rCl}
\gal(K_n) & \cong & \widetilde{\left(A^{(m+1)}\right)_n}\times \prod_{t=n+1}^{\infty}\ker\left(\widetilde{\left(A^{(m+1)}\right)_{(p_t)}} \ra (A^{(m)})_{(p_t)}\right) \\
 &\cong & \widetilde{\left(A_f\right)_n} \times \prod_{t=n+1}^{\infty}\ker\left(\widetilde{(A_f)_{(q_t)}}\ra (A_f)_{(q_t)} \right) 
 \times \Zz_{P_n}\times \prod_{t=n+1}^{\infty}p_t^{m+1}\Zz_{p_t} \\
 &\cong & \widetilde{\left(A_f\right)_n} \times \prod_{t=n+1}^{\infty}\ker\left(\widetilde{(A_f)_{(q_t)}}\ra (A_f)_{(q_t)} \right)
 \times \prod_{t=1}^{\infty}\Zz_{p_t},
\end{IEEEeqnarray*}
where $\Zz_{P_n}$ denotes $\Zz_{p_1}\times \ldots \times \Zz_{p_n}$ (by \cite[Corollary 22.7.8]{FrJa}, for each $p\in \Pp$ we have $\widetilde{C_p}=\Zz_p$, hence $\widetilde{C_{P_n}}=\Zz_{P_n}$). We have the following commutative diagram with exact rows (generalizing the one from Example \ref{cpinftyex}):
\begin{equation*}
 \xymatrix{1 \ar[r]^{}  &  \gal(K) \ar[rr]^{<} & &   \gal(K_n)  \ar[rr]^{\mathrm{res}\ \ \ \ }  &  &  \left(A^{(m+1)}\right)_n \ar[r]^{} & 1 \\
1 \ar[r]^{}  &  \gal(K) \ar[u]_{=}\ar[rr]^{<} & &  \gal(K_n') \ar[u]_{<} \ar[rr]^{\mathrm{res}\ \ \ \ } &  &  P_n\left(A^{(m+1)}\right)_n  \ar[u]_{<}  \ar[r]^{}  & 1.}
\end{equation*}
Therefore, we obtain:
\begin{IEEEeqnarray*}{rCl}
\gal\left(K_n'\right) & = & \mathrm{res}^{-1}\left(P_n(A^{(m+1)})_n\right) \\
 &\cong & \widetilde{\left(A_f\right)_n}\times \prod_{t=n+1}^{\infty}\ker\left(\widetilde{(A_f)_{(q_t)}}\ra (A_f)_{(q_t)} \right) \times P_n\Zz_{P_n}\times \prod_{t=n+1}^{\infty}p_t^{m+1}\Zz_{p_t}\\
 &\cong & \widetilde{\left(A_f\right)_n} \times \prod_{t=n+1}^{\infty}\ker\left(\widetilde{(A_f)_{(q_t)}}\ra (A_f)_{(q_t)} \right)
 \times \prod_{t=1}^{\infty}\Zz_{p_t}\\
  &\cong &     \widetilde{\left(A^{(m)}\right)_n}\times \prod_{t>n}\ker\left(\widetilde{\left(A^{(m)}\right)_{(p_t)}} \ra (A^{(m)})_{(p_t)}\right).
\end{IEEEeqnarray*}
By Theorem \ref{axioms}(3), we get that $K|_{L_{A^{(m)}}}\models A^{(m)}-\mathrm{TCF}$.

For the moreover part, we need to show that the theory $A-\mathrm{TCF}$ is strictly simple (that is: simple, not stable, and not supersimple) for $A$ infinite. For simplicity, by Proposition \ref{mcfact} it is enough to show that each theory $A^{(m)}$-TCF is simple. We use \cite[Corollary 4.31]{Hoff3}, which (very conveniently for us) says that for any group $G$, if the theory $G-\mathrm{TCF}$ exists, then it is simple if and only if the underlying fields of its models are bounded. Let us take $K\models A^{(m)}-\mathrm{TCF}$. By Theorem \ref{axioms}, we have:
$$\gal(K)\cong \prod_{p\in \Pp}\ker\left(\widetilde{\left(A^{(m)}\right)_{(p)}} \ra \left(A^{(m)}\right)_{(p)}\right).$$
Each universal Frattini cover above is a small profinite group being finitely generated. Hence each kernel above is small as well being an open subgroup of a small profinite group. Therefore, $\gal(K)$ is small, since all its pro-$p$ components are small. As a result, the field $K$ is bounded and the theory $A-\mathrm{TCF}$ is simple. Since $K$ is PAC and not separably closed (if $A\neq 0$), by \cite[Fact 2.6.7]{Kimsim} the theory of the pure field $K$ is not stable, so the theory $A-\mathrm{TCF}$ is also not stable. To see that $A-\mathrm{TCF}$ is not supersimple (if $A$ is infinite), it is enough to look at any strictly increasing sequence of finite subgroups of $A$ and consider the corresponding strictly decreasing tower of definable subfields of invariants of $K$.
\end{proof}
\begin{remark}
\begin{enumerate}
\item  Example \ref{cpinftyex} can be generalized in the following way. Let us take $A$ satisfying the equivalent conditions from Remark \ref{equivcond} and let $K\models A-\mathrm{TCF}$. From the above proof of the right-to-left implication in Theorem \ref{superthm}, we have that:
    $$A-\mathrm{TCF} = \bigcup_{m=1}^{\infty} A^{(m)}-\mathrm{TCF}.$$
In particular, $K\models A^{(1)}-\mathrm{TCF}$ and the description of $\gal(K)$ comes from Theorem \ref{axioms}(3), that is:
    $$\gal(K)\cong \ker\left(\widetilde{\widehat{B}}\ra \widehat{B}\right),$$
    where we have:
    $$B:=A_f\oplus \bigoplus_{p\in P_{\infty}}C_p$$
    for $A_f$ and $P_{\infty}$ defined as in the beginning of the proof of Theorem \ref{superthm} in this section.

\item By expressing $A$ as
$$A=\coli_{m,n}\left(A^{(m)}\right)_n,$$
where again the finite subgroups $\left(A^{(m)}\right)_n$ come from the beginning of the proof of Theorem \ref{superthm} in this section, we see that
the theory $A$-TCF is axiomatised by Galois axioms in the sense of Definition \ref{galaxioms}.
\end{enumerate}
\end{remark}
\begin{example}\label{extraq}
We can give now several examples of existentially closed $A$-fields. The ones from Items $(2)$ and $(3)$ below are in the spirit of (but, of course, much easier than) Hrushovski's ``non-standard Frobenius'' from \cite{HrFro}. Let us define:
$$C_{\Pp}:=\bigoplus_{p\in \Pp}C_p,\ \ \ \ \Pp=(p_i)_{i>0},\ \ \ \ P_n:=p_1\ldots p_n.$$
\begin{enumerate}
\item A $C_{\Pp}$-field is a field $K$ with a collection of automorphisms $(\sigma_p)_{p\in \Pp}$ such that for all $p,q\in \Pp$, we have $\sigma_p\sigma_q=\sigma_q\sigma_p$; and for each $p\in \Pp$, we have $\sigma_p^p=\id$. A $C_{\Pp}$-field is strict if and only if for all $p\in \Pp$, we have $\sigma_p\neq \id$. By Theorem \ref{axioms}, it is easy to see that if $K$ is a strict and perfect $C_{\Pp}$-field, then $K$ is e.c. if and only if all the fields of constants
    $$K_n:=\mathrm{Fix}\left(\sigma_{p_1}\right)\cap \ldots \cap \mathrm{Fix}\left(\sigma_{p_n}\right)$$
    are pseudofinite.

  \item Let $q\in \Pp$ and for each $n>0$, we define the following $C_{\Pp}$-field:
  $$\mathbf{K}_{q,n}:=\left(\Ff_{q^{P_n}};\fr^{P_n/p_1},\ldots,\fr^{P_n/p_n},\id,\id,\ldots\right).$$
  Note that each $C_{\Pp}$-field $\mathbf{K}_{q,n}$ is bi-interpretable with the difference field $(\Ff_{q^{P_n}},\fr)$. Let $\mathcal{U}$ be a non-principal ultrafilter on the set of positive integers and
we define the following $\Cc_{\Pp}$-field:
  $$\mathbf{K}_{q}:=\left. \prod_{n>0}\mathbf{K}_{q,n} \middle/  \mathcal{U}. \right.$$
 By {\L}o\'{s} Theorem and Item $(1)$, $\mathbf{K}_{q}$ is an existentially closed $C_{\Pp}$-field of characteristic $q$.

\item Let us define:
      $$\mathbf{K}_{0}:=\left. \prod_{n>0}\mathbf{K}_{p_n,n} \middle/  \mathcal{U}, \right.$$
      where each $C_{\Pp}$-field $\mathbf{K}_{p_n,n}$ comes from Item $(2)$ above.
Similarly as in Item $(2)$, $\mathbf{K}_{0}$ is an existentially closed $C_{\Pp}$-field of characteristic $0$.

\item Let us take $q\in \Pp$ and define:
$$\mathcal{H}:=\prod_{p\in \Pp}p\Zz_p<\widehat{\Zz}\cong \gal(\Ff_q),\ \ \ \ K:=\left(\Ff_q^{\alg}\right)^{\mathcal{H}}.$$
Then, we have:
$$\gal(K/\Ff_q)\cong \widehat{\Zz}/\mathcal{H}\cong \prod_{p\in \Pp}C_p \cong \widehat{C_{\Pp}}.$$
Hence, $K$ becomes naturally a $C_{\Pp}$-field. It is clear from Item $(1)$ that $K$ is an e.c. $C_{\Pp}$-field.
\end{enumerate}
\end{example}

\begin{remark}
\begin{enumerate}
\item We note that if we consider in Example \ref{extraq}(2) the bi-interp\-retable difference field $(\Ff_{q^{P_n}},\fr)$ and take the ultraproduct in the language of difference fields, then the result is completely different: we would obtain a pseudofinite field of characteristic $q$ with the Frobenius automorphism.

\item However, if we do the same in Example \ref{extraq}(3), then we get a pseudofinite field of characteristic $0$ with an automorphism of an infinite order, which should be \emph{generic} in some sense. More precisely, we consider the following difference fields:
$$\mathbf{K}_n:=\left(\Ff_{q_n},\fr\right),\ \ \ \ q_n:=p_n^{P_n}=p_n^{p_1\ldots p_n}$$
and their non-principal ultraproduct.
%Such ultraproducts seem to be worth studying. Difference fields with underlying fields being pseudofinite
Very similar difference fields were considered in \cite{Zou}, that is: the difference fields from \cite{Zou} are also ultraproducts of finite Frobenius difference fields, but the order of growth of the cardinality of the finite fields in \cite{Zou} seems to be much faster than in our case.
%\cite{Rytenphd} and ..., but we the above situation was not taken into account in these works (models of the theories $\mathrm{PSF}_{(m,n,p)}$ are studied %in \cite{Rytenphd}, but they are of positive characteristic and without transcendence difference elements).
%Such difference fields were also studied in \url{https://www.worldscientific.com/doi/10.1142/S0219061319930012} \textbf{retracted!}
%\\
%but we could not find this particular difference field considered...
%\\
%{\bf Check again!!}
%\item Regarding the Galois axioms of the theory $C_{\Pp}-\mathrm{TCF}$ discussed in Example \ref{extraq}(1), it may be enough to assume that the ground %field $K$ is pseudofinite, since then the field of constants $K_n$ are possibly automatically pseudofinite as well. This is why we ask the following %question below.
\end{enumerate}
\end{remark}
%\begin{question}
%Suppose that $K\subseteq L$ is a finite Galois extension and $L$ is a pseudofinite field. Is $K$ a pseudofinite field as well?
%\end{question}

\section{Miscellaneous results}\label{secmisc}
In this section, we collect some results about model theory of group actions on fields, which did not fit to the course of the proof of the main result of this paper (Theorem \ref{superthm}). More precisely, we:
\begin{itemize}
\item provide another argument for the non-existence of the theory $A$-TCF for certain torsion Abelian groups $A$,

\item describe what we are able to prove for groups containing $\Zz\times \Zz$,

\item discuss briefly the case of non-torsion Abelian groups.
\end{itemize}

\subsection{Another negative argument}
In this subsection, we present briefly a different negative argument in the special case of $A=C_p^{(\omega)}$. This was our original argument and we think that it may have an independent interest. Intuitively, the crucial property implying the non-existence of the theory $A-\mathrm{TCF}$ here is some kind of an \emph{auto-duplication} of $A$ inside $A$, i.e. $A$ has a proper subgroup, which is isomorphic to $A$.

We present $C_p^{(\omega)}$ as $\coli C_p^{n}$ and for a $C_p^{(\omega)}$-field $K$, we set as usual $K_n:=K^{C_p^{n}}$.
Let $C=\bigcap_n K_n$ be the field of absolute constants of $K$. By a compactness argument and the Galois theory, we obtain the following.
\begin{lemma}\label{algai}
Suppose that $K$ is a strict $C_p^{(\omega)}$-field, which is $\aleph_0$-saturated. Then, there are $a_1,a_2,\ldots \in K$ such that for each $n$, we have:
\begin{itemize}
\item $K=K_n(a_1,\ldots,a_n)$;
  %\item $K=C(a_1,a_2,\ldots)$
  %\\
  %{\bf this one I am not sure about, but may be not necessary!};
\item $[C(a_n):C]=p$;

\item the extension $C\subseteq C(a_1,\ldots,a_n)$ is Galois and
$$\gal(C(a_1,\ldots,a_n)/C)\cong C_p^n.$$
\end{itemize}
\end{lemma}
%From that we would get (universal Frattini cover argument) that:
%$$\gal(K/C)=C_p^{\times \omega},\ \ \ \ \gal(C)=\widehat{F_{\omega}(p)}.$$
%This should help to show the following.
Using Lemma \ref{algai}, we can show the following improvement of Lemma \ref{contepi} in this case (note that the profinite completion  $ C_p^{\times \omega}=\widehat{C_p^{(\omega)}}$ is not small).
\begin{prop}\label{faithgal}
Suppose that $K$ is an $\aleph_0$-saturated and strict $C_p^{(\omega)}$-field. Consider the following commutative diagram from Lemma \ref{contepi}:
\begin{equation*}
 \xymatrix{C_p^{\times \omega} \ar[rr]^{\alpha\ \ \ \ \ \ }  &  &  \gal(C^{\alg}\cap K/C) \\
C_p^{(\omega)} \ar[u]_{\iota} \ar[rr]^{\varphi\ \ \ \ \ \ } &  &   \aut(K/C) \ar[u]_{\mathrm{res}}.}
\end{equation*}
Then the map $\alpha$ is an isomorphism of profinite groups.
\end{prop}
\begin{proof}
We take $a_1,a_2,\ldots \in K$ given by Lemma \ref{algai}. Then, for each $n>0$ the extension $C\subset C(a_1,\ldots,a_n)$ is Galois with Galois group being  naturally isomorphic to $C_p^n$. Hence, we obtain the following commutative diagram originating from Lemma \ref{contepi}:
\begin{equation*}
 \xymatrix{  &  & C_p^{\times \omega} \ar[lldd]_{s} \ar[rr]^{\alpha\ \ \ \ \ \ }  &  &  \gal(C^{\alg}\cap K/C) \ar[rrdd]^{\mathrm{res}} &  &  \\
  &  &  C_p^{(\omega)} \ar[u]_{\iota} \ar[rr]^{\varphi\ \ \ \ \ \ } &  &   \aut(K/C) \ar[u]_{\mathrm{res}}  \ar[rrd]^{\mathrm{res}} &  &  \\
  C_p^n  \ar[rrrrrr]^{t_n}  \ar[urr]_{\subset} &  &   &  &   &  & \gal(C(a_1,\ldots,a_n)/C)
  ,}
\end{equation*}
where $s$ is the section of the inclusion map $C_p^n\to C_p^{\times \omega}$ and $t_n$ is an isomorphism. Let $\varphi_n$
denote the following composition map:
\begin{equation*}
 \xymatrix{\gal(C^{\alg}\cap K/C) \ar[rr]^{\mathrm{res}\ \ } &  &   \gal(C(a_1,\ldots,a_n)/C) \ar[rr]^{\ \ \ \ \ \ \ \ \ \ \ \ \ \ t_n^{-1}} &  &  C_p^n.}
\end{equation*}
Then the map:
$$\li_n(\varphi_n):\gal(C^{\alg}\cap K/C)\ra C_p^{\times \omega}$$
is the inverse map to the map $\alpha$ from Lemma \ref{contepi}.
\end{proof}
We need the following general result, which is rather obvious.
\begin{lemma}\label{isoth}
Let $\Phi:G\to H$ be an isomorphism of groups and assume that $G-\mathrm{TCF}$ exists. Then $H-\mathrm{TCF}$ exists and we have:
$$H-\mathrm{TCF}=L_{\Phi}(G-\mathrm{TCF}),$$
where $L_{\Phi}$ is the natural bijection (induced by $\Phi:G\to H$) between the set of all the $L_G$-sentences and the set of all the $L_H$-sentences.
\end{lemma}
The next result uses the ``auto-duplication'' idea alluded to in the beginning of this subsection.
\begin{prop}\label{ecmax}
Suppose that $C_p^{(\omega)}-\mathrm{TCF}$ exists. Then, for any existentially closed $C_p^{(\omega+1)}$-field, its obvious reduct is an existentially closed $C_p^{(\omega)}$-field.
\end{prop}
\begin{proof}
For each $i>0$, let $L_i:=L_{\mathrm{rings}}\cup \{\sigma_1,\ldots,\sigma_i\}$ be the language of $C_p^i$-fields. Let $L_{\omega}:=\bigcup_iL_i$ be the language of $C_p^{(\omega)}$-fields, and $L_{\omega+1}:=L_{\omega}\cup \{\sigma_0\}$ be the language of $C_p^{(\omega+1)}$-fields.

For each $n>0$, we have the following group isomorphisms:
$$s:C_p^{(\omega)}\ra C_p^{(\omega)}\times C_p,\ \ s(\sigma_i)=\sigma_{i-1};$$
$$c_n:C_p^{(\omega)}\ra C_p^{(\omega)},\ \ \sigma_0\mapsto \sigma_1,\sigma_1\mapsto \sigma_2,\ldots,\sigma_n\mapsto \sigma_{n+1},\sigma_{n+1}\mapsto \sigma_{0},\sigma_{>n+1}\mapsto \sigma_{>n+1}.$$
There are the corresponding bijections of languages:
$$L_s:L_{C_p^{(\omega)}}\ra L_{C_p^{(\omega)}\times C_p},\ \ \ \ \ \ \ \ L_{c_n}:L_{C_p^{(\omega)}}\ra L_{C_p^{(\omega)}}.$$

By Lemma \ref{isoth} and our assumption, the theory $C_p^{(\omega+1)}-\mathrm{TCF}$ exists. It is enough to show that:
$$C_p^{(\omega)}-\mathrm{TCF}\subset \left(C_p^{(\omega)}\times C_p\right)-\mathrm{TCF}.$$

Take $\psi\in C_p^{(\omega)}-\mathrm{TCF}$. Then, there is $n>0$ such that $\psi\in L_n$. By Lemma \ref{isoth}, $L_s(\psi)\in C_p^{(\omega+1)}-\mathrm{TCF}$. By Lemma \ref{isoth} again, $L_{c_n}(L_s(\psi))\in C_p^{(\omega+1)}-\mathrm{TCF}$. Since $L_{c_n}(L_s(\psi))=\psi$, the result follows.
\end{proof}
\begin{theorem}\label{nomc}
The theory $C_p^{(\omega)}-\mathrm{TCF}$ does not exist.
\end{theorem}
\begin{proof}
Suppose that the theory $C_p^{(\omega)}-\mathrm{TCF}$ exists and we will reach a contradiction.
By Proposition \ref{ecmax}, there is a $C_p^{(\omega+1)}$-field $K$ such that
$$K\models C_p^{(\omega+1)}-\mathrm{TCF},\ \ \ \ \ K|_{L_{\omega}}\models C_p^{(\omega)}-\mathrm{TCF}.$$
Let us set:
$$C:=K^{C_p^{(\omega)}},\ \ \ \ \ C'=K^{C_p^{(\omega+1)}}.$$
Then, we have $\gal(C/C')=C_p$ and:
$$[\gal(C'):\gal(C)]=p=[\gal(C^{\alg}\cap K/C'):\gal(C^{\alg}\cap K/C)].$$
By Proposition \ref{faithgal}, we also have:
$$\gal(C^{\alg}\cap K/C')\cong C_p^{\times (\omega+1)}.$$
Hence $\gal(C^{\alg}\cap K/C)$ is a codimension one $\Ff_p$-subspace of $C_p^{\times (\omega+1)}$ and it may be identified with $C_p^{\times \omega}$.

By Lemma \ref{fracov}, we have the following isomorphism:
$$\Psi:\gal(C')\ra \widehat{F_{\omega+1}(p)},$$
and $\Psi\left(\gal(C)\right)=\widehat{F_{\omega}(p)}$, which is a contradiction, since the index $[\widehat{F_{\omega+1}(p)}:\widehat{F_{\omega}(p)}]$ is infinite.
\end{proof}

\subsection{Groups containing $\Zz\times \Zz$}\label{secudi}
Let $G$ be an arbitrary group. The notion of a p.e.c. $G$-field, which we discussed in Section \ref{pecsec}, was pointed out to us by Udi Hrushovski in an attempt to show the following.
\begin{conj}\label{zzconj}
If $G$ has a subgroup isomorphic to $\Zz\times \Zz$, then the theory $G-\mathrm{TCF}$ does not exist.
\end{conj}
We discuss below a strategy for a proof of Conjecture \ref{zzconj}, explain where is the problem with this strategy, and give a weaker statement, which can be still proved using this strategy.

Firstly, Hrushovski's proof of the non-existence of $(\Zz\times \Zz)$-TCF (see \cite{Kikyo1} and \cite[Section 5.1]{BK}) gives a stronger result, which we formulate below.
\begin{theorem}[Hrushovski]\label{hrthm}
There is no $\aleph_0$-saturated and p.e.c $(\Zz\times \Zz)$-field $K$ such that:
\begin{itemize}
\item the primitive third root of unity $\zeta_3$ belongs to $K$;

\item we have:
$$(1,0)\cdot \zeta_3=\zeta_3^2,\ \ \ \ \ (0,1)\cdot \zeta_3=\zeta_3^2.$$
\end{itemize}
\end{theorem}
The hope was that if (in the situation of Conjecture \ref{zzconj}) $K$ is an $\aleph_0$-saturated and e.c (even just p.e.c.) $G$-field, then its reduct to $\Zz\times \Zz$ would contradict Theorem \ref{hrthm} using Proposition \ref{udipec}. However, the problem is that the induced action of $\Zz\times \Zz$ on $K$ need not satisfy the conditions from Theorem \ref{hrthm}. Actually, if $G=\Qq\times \Qq$, then this induced action \emph{never} satisfies the conditions from Theorem \ref{hrthm}.

It is easy to give algebraic conditions on the group $G$ yielding the conditions from Theorem \ref{hrthm}, which we do below.
\begin{prop}\label{indexpr}
Suppose that:
\begin{itemize}
  \item there is $H\leqslant G$ such that $H\cong \Zz\times \Zz$,
  \item there is $N<G$ of index $2$ such that $H\nsubseteq N$.
\end{itemize}
Then, the theory $G-\mathrm{TCF}$ does not exist.
\end{prop}
\begin{proof}
Assume that the theory $G-\mathrm{TCF}$ exists and we will reach a contradiction. By our assumptions, there is a group homomorphism $\varphi:G\to C_2$ such that $\varphi(H)=C_2$. Since we have
$$\gal(\Qq(\zeta_3)/\Qq)\cong C_2,$$
the above homomorphism $\varphi$ gives $\Qq(\zeta_3)$ the $G$-field structure such that the reduct of this structure to $H$ satisfies the conditions from Theorem \ref{hrthm}. Since the theory $G-\mathrm{TCF}$ exists, we can extend the $G$-field $\Qq(\zeta_3)$ to an existentially closed and $\aleph_0$-saturated $G$-field $K$, and in this case the strategy described above works giving a contradiction with Theorem \ref{hrthm}.
\end{proof}
We give below an explicit (although, looking a bit strange) statement, which easily follows from Proposition \ref{indexpr}.
%One can easily produce another statements, which are similar to Corollary \ref{gc2} below.
\begin{cor}\label{gc2}
If $\Zz\times \Zz$ embeds into $G$, then the theory $(G\times C_2)-\mathrm{TCF}$ does not exist.
\end{cor}
\begin{proof}
We take $G$ to play the role of $N$ from Proposition \ref{indexpr} and we also set:
$$H:=\langle ((1,0),a),((0,1),a)\rangle,$$
where we identify $\Zz\times \Zz$ with its image in $G$ and take $a$ as the generator of $C_2$.
\end{proof}

\subsection{Arbitrary Abelian groups}
It is tempting to extend Theorem \ref{superthm} to the case of an arbitrary Abelian group. However, there are the following problems.
\begin{enumerate}
\item On the negative argument side, it is not clear even how to show that the theory $(\Qq\times \Qq)-\mathrm{TCF}$ does not exist, as was pointed out in Section \ref{secudi}.

\item Regarding the positive argument side, it may be also unclear how to deal e.g. with the case of $C_{p^{\infty}}\times \Qq$. We can present this group as:
    $$C_{p^{\infty}}\times \Qq=\coli_n C_{p^n}\times \frac{1}{n!}\Zz,$$
    but, since the group $C_{p^n} \times \Zz$ is not finite, we can not use only the Galois axioms and we should also consider direct limits with respect to the Bass-Serre theory (see \cite{BK}). This may be doable, but such methods thematically do not fit to this paper and this circle of topics will be picked up elsewhere.
\end{enumerate}
%{\bf Can we say anything about \emph{locally finite nilpotent} groups??}
\emph{Acknowledgements}.

\vspace{2mm}

We are very thankful to the referee for a careful and useful report.

We would like to thank Daniel Hoffmann, Udi Hrushovski, and Nick Ramsey for their comments. The second author would like to thank the members of the model theory group in Wroc{\l}aw for their constructive remarks during model theory seminars at Wroc{\l}aw University.

 This work was partially done in the Nesin Mathematics Village (\c{S}irince,
Izmir, Turkey), and we would like to thank the village for its hospitality.
\bibliographystyle{plain}
\bibliography{harvard}

\end{document}